\documentclass[12pt]{amsart}
\usepackage{amssymb}
\usepackage{amsbsy}
\usepackage{amscd}
\usepackage{comment}
\usepackage{enumerate}
\usepackage{mathrsfs}
\usepackage{eucal}

\makeatletter

\renewcommand{\thesubsection}{\thesection(\@roman\c@subsection)}

\makeatother

\usepackage{verbatim}
\usepackage{version}
\usepackage{color}

\excludeversion{NB}

\newtheorem{Theorem}[equation]{Theorem}
\newtheorem{Corollary}[equation]{Corollary}
\newtheorem{Lemma}[equation]{Lemma}
\newtheorem{Proposition}[equation]{Proposition}

\theoremstyle{definition}
\newtheorem{Definition}[equation]{Definition}

\theoremstyle{remark}
\newtheorem{Remark}[equation]{Remark}

\numberwithin{equation}{section}

\newcommand{\thmref}[1]{Theorem~\ref{#1}}
\newcommand{\secref}[1]{\S\ref{#1}}
\newcommand{\lemref}[1]{Lemma~\ref{#1}}
\newcommand{\propref}[1]{Proposition~\ref{#1}}

\newcommand{\subsecref}[1]{\S\ref{#1}}

\newcommand{\defref}[1]{Definition~\ref{#1}}
\newcommand{\remref}[1]{Remark~\ref{#1}}

\newcommand{\defeq}{\overset{\operatorname{\scriptstyle def.}}{=}}

\newcommand{\C}{{\mathbb C}}
\newcommand{\Z}{{\mathbb Z}}

\newcommand{\algsl}{\operatorname{\mathfrak{sl}}}

\newcommand{\g}{{\mathfrak g}}
\newcommand{\h}{{\mathfrak h}}

\newcommand{\End}{\operatorname{End}}

\newcommand{\Ker}{\operatorname{Ker}}

\newcommand{\ve}{\varepsilon}
\renewcommand{\MR}[1]{}

\newcommand{\ad}{{\operatorname{ad}}}
\newcommand{\re}{{\operatorname{re}}}

\usepackage[bookmarks=false]{hyperref}

\begin{document}
\author{Nicolas Guay, Hiraku Nakajima, Curtis Wendlandt}
\title[Coproduct for Yangians of affine Kac-Moody algebras]
{Coproduct for Yangians of affine Kac-Moody algebras
}
\address{N.G.: Department of Mathematical and Statistical Sciences,
University of Alberta, CAB 632, Edmonton, AB  T6G 2G1, Canada}
\email{nguay@ualberta.ca}
\thanks{N.G.: Supported by a Discovery Grant of the Natural Sciences and Engineering Research Council of Canada}
\address{H.N.: Research Institute for Mathematical Sciences,
Kyoto University, Kyoto 606-8502, Japan and Kavli Institute for the Physics and Mathematics of the Universe (WPI), The University of Tokyo, Kashiwa 277-8583, Japan}
\email{nakajima@kurims.kyoto-u.ac.jp, hiraku.nakajima@ipmu.jp}
\thanks{H.N.: Supported by the Grant-in-aid
for Scientific Research (No.25220701, No.16H06335), JSPS, Japan.
}
\address{C.W.: Department of Mathematical and Statistical Sciences,
University of Alberta, CAB 632, Edmonton, AB  T6G 2G1, Canada}
\email{cwendlan@ualberta.ca}
\thanks{C.W.: Supported by a Doctoral Scholarship (CGS-D) of the Natural Sciences and Engineering Research Council of Canada
}

\begin{abstract}
Given an affine Kac-Moody algebra, we explain how to construct a coproduct on its associated Yangian.  In order to prove that this coproduct is an algebra homomorphism, we obtain, in the first half of this paper, a minimalistic presentation of the Yangian when the Kac-Moody algebra is, more generally, symmetrizable.
\end{abstract}

\maketitle
\tableofcontents

\section{Introduction}

The quantized enveloping algebra $U_{\hbar}(\g_0)$ of a simple Lie algebra $\g_0$ is a Hopf algebra which provides a quantization of a certain Lie bialgebra structure on $\g_0$. Being a Hopf algebra, it not only possesses an associative product, but is also equipped with a coproduct. This is what distinguishes it from the enveloping algebra $U(\g_0)$ of $\g_0$ because, as algebras, $U_{\hbar}(\g_0)$ is actually a trivial deformation of $U(\g_0)$. (This is a consequence of the vanishing of the second Hochschild cohomology group of $U(\g_0)$ - see Theorem XVIII.4.1 in  \cite{Kas-book}.) The definition of the Drinfeld-Jimbo quantized enveloping algebra can be extended to any symmetrizable Kac-Moody algebra. Furthermore, using what is commonly referred to as Drinfeld's second realization \cite{DrinfeldNew}, it is even possible to define affinizations of quantized Kac-Moody algebras \cite{Na01,He1} (see earlier references therein). These include, in particular, quantum toroidal algebras.

There are two families of quantized enveloping algebras of affine type: the Drinfeld-Jimbo quantum affine algebras $U_{\hbar}(\g)$ and the Yangians $Y_{\hbar}(\g_0)$. (Here, $\g$ is the affine Lie algebra corresponding to $\g_0$.) Although a priori quite different, there exist completions of these algebras which are in fact isomorphic \cite{GTL1} (see also \cite{GuMa} for the proof of a weaker result). Furthermore, tensor equivalences between categories of representations of these two quantum groups have been established in \cite{GTL2,GTL3}.  It is also possible to associate a Yangian $Y_{\hbar}(\g)$  to any symmetrizable Kac-Moody algebra $\g$, in particular to an affine Lie algebra: one thus obtains an affine Yangian. Quantum toroidal algebras and affine Yangians are two of the main examples of quantized enveloping algebras of double affine type, a third example being provided by the deformed double current algebras \cite{Gu2,Gu3,GuYa,TLY}. 

For both quantum affine algebras and Yangians there is a standard coproduct: in the former case, it is the coproduct given in terms of the Kac-Moody generators (as in \cite[Definition-Prop. 6.5.1]{CP-book}), while in the latter case it is the coproduct given in terms of the generators $\{X,J(X)\}_{X\in \g_0}$ (as in \cite[Theorem 12.1.1]{CP-book}). There also exist non-standard coproducts on these two families which are originally due to V. Drinfeld - see Definition 3.2 in \cite{DiFr} and \S 6 in \cite{DiKh}. Actually, the authors of \cite{DiKh} need to consider the double of the Yangian, but it is also possible to degenerate the non-standard coproduct on quantum affine algebras to obtain one on the Yangian itself - see \cite{GTL3}. These have also appeared in the recent work \cite{YaZh} via an isomorphism between the Yangian and a cohomological Hall algebra which turns out to be an isomorphism of bialgebras when the Yangian is equipped with the non-standard coproduct and the cohomological Hall algebra is equipped with the comultiplication constructed in \textit{loc.\ cit.}  For quantum affine algebras, the non-standard and standard coproducts are related by a meromorphic twist (see \cite[Proposition 3.8]{EKP}, \cite[Theorem 3.1]{KoTo} and \cite{GTL3}), and a rational version of this result is expected to hold for Yangians \cite{GTL3}. (For the Yangian double, see \cite[Proposition 5.1]{EnFe} and \cite[Remark 5]{En}.) The non-standard coproducts are not exactly genuine  coproducts as they involve infinite sums and map into certain completed tensor products: see \cite{He1} and \cite{GTL3}. Additionally, they cannot always be used to define a module structure on tensor products of two modules because of convergence issues. The definitions of these non-standard coproducts extend automatically to quantum toroidal algebras and affine Yangians. In this context, they were used in the work of D. Hernandez \cite{He1,He2} and of B. Feigin, E. Feigin, M. Jimbo, T. Miwa and E. Mukhin \cite{FFJMM1,FFJMM2,FJMM}. The papers \cite{FFJMM1,FFJMM2} are about the quantum toroidal algebra of $\mathfrak{gl}_1$ (which was however given the name quantum continuous $\mathfrak{gl}_{\infty}$). There is also an affine Yangian of type $\mathfrak{gl}_1$ studied, for instance, in \cite{Ts,BeTs}. This affine Yangian was shown in \cite{ArSc} to be isomorphic to a certain algebra $\mathrm{\mathbf{SH}}^{\mathbf{c}}$ which is a sort of stable limit of spherical trigonometric Cherednik algebras of type $\mathfrak{gl}_{\ell}$ and was introduced in \cite{ScVa} where it was used to prove a version of the AGT-conjecture. The algebra $\mathrm{\mathbf{SH}}^{\mathbf{c}}$, and thus the affine Yangian of $\mathfrak{gl}_1$, admits a topological coproduct which is close to the standard coproduct: see Theorem 7.9 in \cite{ScVa}. It is not clear that the proof in \cite{ScVa} that the coproduct is well-defined can be modified for general $Y_{\hbar}(\g)$.  We will not consider the affine Yangian of type $\mathfrak{gl}_1$ in the present paper.  

The goal of the present paper is to introduce a coproduct $\Delta$ on affine Yangians which is a natural analog of the standard coproduct on Yangians of finite-dimensional simple Lie algebras. We first define it via the action of the affine Yangian on the tensor product of two modules in the category $\mathscr{O}$ (\defref{def:Delta}) and prove that it is an algebra homomorphism (\thmref{thm:coproduct}). (Our proof also works for $Y(\g_0)$; in this case, \thmref{thm:coproduct} is, of course, already known, but a proof has never appeared in the literature.) In the subsequent section (\secref{Sec:comp}), we introduce a completion of the tensor product of the affine Yangian with itself and explain how $\Delta$ can be viewed as an algebra homomorphism from the affine Yangian into that completion: see \propref{Prop:Del.comp}. That completion is defined using a grading which is not compatible with the algebra structure on $Y_{\hbar}(\g) \otimes Y_{\hbar}(\g)$, so an argument is needed to prove that the multiplication on $Y_{\hbar}(\g) \otimes Y_{\hbar}(\g)$ extends to it (see \propref{prop:prodwhY}). Furthermore, in \secref{Sec:paracoprod}, we introduce a third version of our coproduct which depends on a formal parameter $u$ and which defines an algebra homomorphism from $Y_{\hbar}(\g)$ to $(Y_{\hbar}(\g) \otimes Y_{\hbar}(\g))(\!(u)\!)$, from which it is possible to recover the algebra homomorphisms of \thmref{thm:coproduct} and \propref{Prop:Del.comp}. It has been used in \cite{GRW2} to give a proof of the PBW property for affine Yangians in the simply laced case. Moreover, in \secref{Sec:comp}, we need to assume two conditions regarding the existence of a triangular decomposition of the Yangian and those are not needed in \secref{Sec:paracoprod}.  One advantage of our coproduct is that it can be used to define a module structure on the tensor product of two modules in the category $\mathscr O$ without any convergence issues.

It is natural to conjecture that this new coproduct $\Delta$ is related to the coproduct alluded to two paragraphs above via a certain twist as suggested in \cite{GTL3}, but it is not at all clear that this is the case because such a twist is expected to coincide with the lower triangular part of the universal $R$-matrix of the Yangian and no universal $R$-matrix is known for the affine Yangians.  At least, in the case when $\g$ is symmetric,  see the last part of the third paragraph below.

In order to prove \thmref{thm:coproduct}, we need to simplify the presentation of the affine Yangians: this is accomplished in \secref{sec:YangianKM} - see \thmref{thm:deduction}. The results of this section are actually valid more generally for Yangians of symmetrizable Kac-Moody algebras which satisfy certain mild conditions. For affine Yangians, these conditions are equivalent to the assumption that $\g$ is not of type $A_1^{(1)}$ or $A_2^{(2)}$. However, we expect that \thmref{thm:coproduct} holds more generally for all affine Lie algebras and even for any symmetrizable Kac-Moody algebra. 

When $\g$ is of affine type $A_{n-1}^{(1)}$, it is possible to introduce an extra parameter $\ve$ in the definition of $Y_\hbar(\g)$ in order to obtain a two parameter Yangian $Y_{\hbar,\ve}(\g)$ (see \defref{def:TwoPar}). All the main results of this paper hold in this greater generality: this is briefly explained in \secref{sec:2para}. These two parameter Yangians have been studied by the first named author in \cite{Gu1,Gu2}. (Quantum toroidal algebras of type $A$ can also depend on two parameters, see \cite{VaVa}.)

When $\g$ is symmetric (including the $\widehat{\mathfrak{gl}}_1$-case), there is a geometric construction of the Yangian using quiver varieties \cite{MaOk} which gives a coproduct as well as the universal $R$-matrix and integrable representations. By the construction in \cite{Va}, we have a homomorphism from $Y_{\hbar}(\mathfrak g)$ to the Yangian in \cite{MaOk}. Our formula \eqref{eq:assign} implies that it is compatible with the coproduct on both Yangians. Since we do not know that it is an isomorphism (or whether it is injective or surjective), \cite{MaOk} does not imply our main result, but it gives evidence that \thmref{thm:coproduct} is true in a more general setting. In the geometric context, the non-standard coproduct corresponds to the restriction to the torus fixed point (see \cite{VV02} for quantum affine algebras; the same proof works for affinizations of quantized Kac-Moody algebras and also for Yangians). Then the coproduct in \cite{MaOk} and the non-standard coproduct are related by the stable envelope, which is a `half' of the universal $R$-matrix, and is a key object introduced in \cite{MaOk}. In particular, the conjecture in the third paragraph above is known for the Yangian in \cite{MaOk}.

In \cite{FKPRW}, the authors define a coproduct on shifted Yangians which is related to the coproduct on $Y_{\hbar}(\g)$ via shift maps: see \S 4.6 in \textit{loc.\ cit.} Their Theorem 4.8 states that this coproduct is well-defined in the sense that it respects the defining relations of the shifted Yangians. The proof of that theorem depends on the main results of our present paper regarding the coproduct $\Delta$ on $Y_{\hbar}(\g)$. 

It is natural to expect that a coproduct similar to the one constructed in the present paper exists for quantum toroidal algebras. It would also be interesting to obtain one for deformed double current algebras as it would certainly be useful to make progress in understanding their largely unexplored representation theory. 

\begin{center}\sc Acknowledgment \end{center} The authors would like to thank Oleksandr Tsymbaliuk and Mamoru Ueda for pointing out an inaccuracy in an earlier version of this paper. They are also grateful to the referee for a very thorough reading of our manuscript.

\section{The Yangian of a Kac-Moody algebra}\label{sec:YangianKM}

Let $\g$ be a symmetrizable Kac-Moody algebra associated with the indecomposable Cartan matrix $(a_{ij})_{i,j\in I}$, where $I$ is the set of vertices of the Dynkin diagram corresponding to $\g$.
We also fix an invariant inner product $(\ ,\ )$ on $\g$.
We normalize the Chevalley generators $x_i^\pm$, $h_i$ so that
$(x_i^+,x_i^-) = 1$ and $h_i = [x^+_i,x^-_i]$. Let $\Delta$ ($\Delta^{\mathrm{re}}$, resp. $\Delta^{\mathrm{im}}$) be the set of all roots of $\g$ (of all real roots, resp. of all imaginary roots), and let $\Delta_{\pm}$ be the sets of positive and of negative roots.  $\Delta^{\mathrm{re}}_{\pm}$ is defined similarly. The simple roots will be denoted $\alpha_i$ for $i\in I$. The root space corresponding to a root $\beta \in \Delta$ will be denoted $\g_{\beta}$. When $\g$ is an affine Lie algebra, we let $\delta$ be the positive imaginary root such that $(\Z\setminus\{ 0 \})\delta$ is the set of all imaginary roots of $\g$ \cite{Ka-book}. Let $\g'$ be the derived subalgebra $[\g,\g]$.

In the definition below, and consequently for the rest of this paper, we will assume that $\g$ is not of type $A_1^{(1)}$: see the definition in \S 1.2 in \cite{BeTs} and Definition 5.1 in \cite{Ko} for the correct definition of the Yangian in this case. All algebras will be defined over $\C$ and will be Lie algebras or associative, unital algebras, unless specified otherwise.

\begin{Definition}\label{def:Yangian} Let $\hbar\in \C$. The Yangian $Y_\hbar(\g')$ is the algebra with generators $x_{i
  r}^\pm$, $h_{i r}$ ($i\in I$, $r\in\Z_{\ge 0}$) subject to the following
defining relations: {\allowdisplaybreaks[4]
\begin{gather}
  [h_{i r}, h_{j s}] = 0, \label{eq:relHH}\\
  [h_{i 0}, x^\pm_{j s} ] = \pm (\alpha_i,\alpha_j) x^\pm_{j s},
  \label{eq:relHX}
\\
  [x^+_{i r}, x_{j s}^-] = \delta_{ij} h_{i, r+s}, \label{eq:relXX}
\\
  [h_{i, r+1}, x^\pm_{j s}] - [h_{i r}, x^\pm_{j, s+1}] =
  \pm\frac{\hbar (\alpha_i,\alpha_j)}2 \left(
    h_{i r} x^\pm_{j s} + x^\pm_{j s} h_{i r}
  \right), \label{eq:relexHX}
\\
  [x^\pm_{i, r+1}, x^\pm_{j s}] - [x^\pm_{i r}, x^\pm_{j, s+1}] = 
  \pm\frac{\hbar (\alpha_i,\alpha_j)}2 \left(
    x^\pm_{i r} x^\pm_{j s} + x^\pm_{j s} x^\pm_{i r}\right),
   \label{eq:relexXX}
\\
  \sum_{\sigma\in S_b}
   [x^\pm_{i r_{\sigma(1)}}, [ x^\pm_{i r_{\sigma(2)}}, \cdots,
       [x^\pm_{i, r_{\sigma(b)}}, x^\pm_{j s}] \cdots ]]
   = 0
   \quad \text{if $i\neq j$,}
 \label{eq:relDS}
\end{gather}
where} $b = 1 - a_{ij}$.
The Yangian $Y_\hbar(\g)$ is the algebra generated by $\{x_{ir}^\pm,h_{ir}\}_{i\in I,r\in \Z_{\geq 0}}\cup \h$, where $\h$ is the Cartan subalgebra of $\g$, subject to the relations of $Y_\hbar(\g^\prime)$ in addition to
\begin{equation}\label{eq:Cartan}
\begin{gathered}
  h_{i 0} = \frac{(\alpha_i,\alpha_i)}2 \alpha_i^\vee \text{ where the simple coroot $\alpha_i^{\vee}$ belongs to $\h$}, \\
  [h,h_{i r}] = 0, \quad
  [h, x^\pm_{i r}] = \langle \alpha_i, h\rangle x^\pm_{i r} \quad
  \text{for $h\in\mathfrak h$}.
\end{gathered}
\end{equation}
\end{Definition}

Given two elements $a,b$ of some algebra $\mathcal{A}$, we set $\{ a,b\} =ab+ba$. In particular, the right-hand sides of \eqref{eq:relexHX} and \eqref{eq:relexXX} could be written in terms of $\{ h_{ir}, x_{js}^{\pm} \}$ and $\{ x_{ir}^{\pm}, x_{js}^{\pm} \}$, respectively. 

Observe that for any pair of non-zero complex numbers $\hbar_1,\hbar_2\in \C^\times$ we have $Y_{\hbar_1}(\g)\cong Y_{\hbar_2}(\g)$. With this in mind, 
we set $\hbar=1$ and denote $Y_{\hbar}(\g)$ simply by $Y(\g)$ hereafter (except in \secref{sec:2para}). Similarly, we denote $Y_\hbar(\g^\prime)$ by $Y(\g^\prime)$. Note that the assignment $x^\pm_i, h \mapsto x^\pm_{i 0}, h$ gives an algebra homomorphism \begin{equation} \iota\colon U(\g)\to Y(\g). \label{eqn:iota} \end{equation}

\subsection{A minimalistic presentation of \texorpdfstring{$Y(\g)$}{}}

In this subsection, we state the first main result of this paper (\thmref{thm:deduction}), which we will prove in \subsecref{Ssec:Pf1} below.

From the defining relations, we can see that $Y(\g')$ is generated by $x^\pm_{i 0}$,
$h_{i 0}$ and $h_{i 1}$ with $i\in I$ (see, for instance, \cite{Lev}). In fact, we can obtain $x^\pm_{i r}$, $h_{i
  r}$ inductively from the relations
\begin{equation}\label{eq:rec}
  \begin{split}
   x^\pm_{i, r+1} &= \pm (\alpha_i, \alpha_i)^{-1}
   [h_{i 1} - \frac12 h_{i 0}^2, x^\pm_{i r}],
\\
   h_{i r} &= [ x^+_{i r}, x^-_{i 0}].
  \end{split}
\end{equation}

To simplify the first of these formulas as well as future computations, we introduce the auxiliary generators $\tilde h_{i 1}$, with  $i\in I$, by setting
\begin{equation}\label{eq:th}
   \tilde h_{i 1} \defeq h_{i 1} - \frac12 h_{i 0}^2.
\end{equation}
Then $x^\pm _{i, r+1} = \pm(\alpha_i,\alpha_i)^{-1} [\tilde h_{i 1},
x^\pm_{i r}]$ and \eqref{eq:relexHX} with $(r,s) = (0,0)$ can be rewritten as
\begin{equation}\label{eq:relexHX2}
  [\tilde h_{i 1}, x^\pm_{j 0}] = \pm (\alpha_i,\alpha_j) x^\pm_{j 1}.
\end{equation}

We want to reduce the number of relations to make it easier to check the compatibility of the coproduct $\Delta$ to be introduced in \secref{Sec:cop}. Such work was done by S. Levendorskii for $\g$ finite-dimensional in \cite{Lev}: See Theorem 1.2 therein. 

\begin{Theorem}\label{thm:deduction}
Suppose that, for any $i,j\in I$ with $i\neq j$, the matrix $\begin{pmatrix} a_{ii} & a_{ij} \\ a_{ji} & a_{jj} \end{pmatrix}$ is invertible and that, in addition, there exists one pair of indices $i,j\in I$ such that $a_{ij}=-1$. Then $Y(\g')$ is isomorphic to the algebra generated by $\{x_{ir}^\pm,h_{ir}\}_{i\in I,r\in \{0,1\}}$, subject only to the relations
{\allowdisplaybreaks[4]
\begin{gather}
  [h_{i r}, h_{j s}] = 0 \quad (0\le r,s\le 1), \label{eq:relHH'}\\
  [h_{i 0}, x^\pm_{j s} ] = \pm (\alpha_i,\alpha_j) x^\pm_{j s} \quad
  (s=0,1),
  \label{eq:relHX'}
\\
  [x^+_{i r}, x_{j s}^-] = \delta_{ij} h_{i, r+s} \quad
  ( 0 \le r+s \le 1),
  \label{eq:relXX'}
\\
  [\tilde h_{i 1}, x^\pm_{j 0}] = \pm (\alpha_i,\alpha_j) x^\pm_{j 1}
  \qquad (\tilde h_{i 1} \defeq h_{i 1} - \frac12 h^2_{i 0}),
  \label{eq:relexHX2'}
\\
  [x^\pm_{i 1}, x^\pm_{j 0}] - [x^\pm_{i 0}, x^\pm_{j 1}] = 
  \pm \frac{(\alpha_i,\alpha_j)}2 \left(
    x^\pm_{i 0} x^\pm_{j 0} + x^\pm_{j 0} x^\pm_{i 0}\right),
   \label{eq:relexXX'}
\\
   \mathrm{ad}(x^\pm_{i 0})^{1-a_{ij}}(x^\pm_{j 0}) = 0
   \quad \text{if $i\neq j$.}
 \label{eq:relDS'}
\end{gather}
}
\end{Theorem}
In this algebra, we also define elements $x_{ir}^{\pm}$ and $h_{ir}$ for $r\ge 2$ using \eqref{eq:rec}.

\begin{Remark}\label{CMinv}  If $\g$ is of affine type, then $\g$ satisfies the conditions of the previous theorem 
provided it is not of type $A^{(1)}_1$ or $A^{(2)}_2$. Indeed, that $(a_{kl})_{k,l\in \{i,j\}}$ is invertible for $i\neq j$ is due to \cite[Proposition 4.7(b)]{Ka-book}, and the
existence of a pair $(i,j)$ such that $a_{ij}=-1$ can be seen by inspection of the corresponding Dynkin diagram (see \S 4.7 of \cite{Ka-book}). 
\end{Remark}
Observe that the statement of Theorem 1.2 in \cite{Lev} is precisely that $Y(\g)$ (where $\g$ is finite-dimensional and simple) is isomorphic 
to the algebra generated by the elements $x_{i0}^{\pm},h_{i0}$ and $x_{i1}^{\pm},h_{i1}$, with $i\in I$, subject to the defining relations \eqref{eq:relHH'}-\eqref{eq:relDS'} 
together with the relation 
\begin{equation}\label{eq:relHH2}
   [[\tilde h_{i 1}, x^+_{i 1}], x^-_{i 1}]
   +  [x^+_{i 1}, [\tilde h_{i 1}, x^-_{i 1}]] = 0.
\end{equation}
Moreover, Levendorskii's argument also applies in the case where $\g$ is a symmetrizable Kac-Moody algebra satisfying the conditions of \thmref{thm:deduction}.

Unfortunately, for our purposes, the relation \eqref{eq:relHH2} is still difficult to work with. In the case where $\g$ is of type $\algsl_{n+1}$ or $\hat\algsl_{n+1}$ with $n\ge 3$, this difficulty was addressed by the first named author in \cite{Gu2} where it was shown that the relation \eqref{eq:relHH2} can be deduced from those given in the statement of \thmref{thm:deduction}.

\subsection{Proof of \thmref{thm:deduction}}\label{Ssec:Pf1}

 As consequence of the remarks made at the end of the previous subsection, to prove \thmref{thm:deduction} it suffices to show that the relation \eqref{eq:relHH2} can be derived from \eqref{eq:relHH'}-\eqref{eq:relDS'}. To prove this, we will proceed as follows: 
First, we establish that some of the relations \eqref{eq:relHH}-\eqref{eq:relDS} can be derived from \eqref{eq:relHH'}-\eqref{eq:relDS'} for certain values of $i,j$ and $r,s$. 
Then, following Levendorskii's argument, we use these relations to establish a sequence of lemmas and propositions which allow us to conclude that \eqref{eq:relHH2} is indeed satisfied for all $i\in I$.

Our first goal is the construction of an element ${\Tilde{h}}_{i 2}$
such that $[{\Tilde{h}}_{i 2},x^\pm_{i 0}] = \pm (\alpha_i,\alpha_i)
x^\pm_{i 2} \pm (\alpha_i,\alpha_i)^3 x^\pm_{i 0}/12$.
This was done in \cite[Cor.~1.5]{Lev}. We reproduce the proof in order to point out that the argument does not use \eqref{eq:relHH2}. From this point on we assume that the Cartan matrix of $\g$ satisfies the assumptions of \thmref{thm:deduction} and that the elements $x_{i0}^{\pm},h_{i0}$ and $h_{i1}$, for $i\in I$, satisfy only relations \eqref{eq:relHH'}-\eqref{eq:relDS'}.

\begin{Lemma}\label{lem:HX} The following relations are satisfied for all $i,j\in I$ and $r\in \Z_{\geq 0}$:
  \begin{equation*}
     [h_{i 0}, x^\pm_{j r}] = \pm(\alpha_i,\alpha_j) x^\pm_{j r},
\qquad
    [\tilde h_{i 1}, x^\pm_{j r}] = \pm(\alpha_i,\alpha_j) x^\pm_{j,r+1}.    
  \end{equation*}
\end{Lemma}

\begin{proof}
  One can show the second equality by induction on $r$. If $r = 0$, it is
  nothing but \eqref{eq:relexHX2'}. The general case follows by using \eqref{eq:rec}, $[\tilde h_{i 1}, \tilde h_{j 1}] = 0$ (which follows immediately from \eqref{eq:relHH'}) and the inductive assumption.
The first equality can be proven in the same way.
\end{proof}

\begin{Lemma}\label{lem:XXii}
  The relation \eqref{eq:relXX} holds when $i=j$, $r+s\le 2$.
\end{Lemma}

\begin{proof}
  From \eqref{eq:relHH'} with $r, s\le 1$ and \eqref{eq:relXX'} with
  $i=j$, $(r,s) = (1,0)$, we have
  \begin{equation*}
    0 = [h_{i 1}, \tilde h_{i 1}] =
    [[x^+_{i 1}, x^-_{i 0}], \tilde h_{i 1}]
    = - (\alpha_i, \alpha_i) 
    \left([x^+_{i 2}, x^-_{i 0}] - 
     [x^+_{i 1}, x^-_{i 1}]\right),
  \end{equation*}
where we have used \lemref{lem:HX} in the last equality. Therefore
\begin{equation*}
  [x^+_{i 1}, x^-_{i 1}] = [x^+_{i 2}, x^-_{i 0}] = h_{i 2}.
\end{equation*}

Similarly we use \eqref{eq:relXX'} with $(r,s) = (0,1)$ instead to get
\(
 h_{i2}= [x^+_{i 1}, x^-_{i 1}] = [x^+_{i 0}, x^-_{i 2}].
\)
\end{proof}

\begin{Lemma}\label{lem:exXXii}
  The relation \eqref{eq:relexXX} holds when $i=j$, $(r,s)=(1,0)$, i.e.,
  \begin{equation}\label{eq:relexXX10}
    [x^\pm_{i 2}, x^\pm_{i 0}]
    = \pm \frac{(\alpha_i,\alpha_i)}2 \left(x^\pm_{i 1} x^\pm_{i 0} + 
     x^\pm_{i 0} x^\pm_{i 1} \right).
  \end{equation}
\end{Lemma}

\begin{proof} This follows immediately by applying $[\tilde h_{i 1},\cdot ]$ to \eqref{eq:relexXX'} with $i=j$.
\end{proof}

\begin{Lemma}\label{lem:HXii}
  The relation \eqref{eq:relexHX} holds when $i=j$, $(r,s) = (1,0)$,
  i.e.,
\begin{equation*}
  [ h_{i 2}, x^\pm_{i 0}] - [h_{i 1}, x^\pm_{i 1}] 
  = \pm \frac{(\alpha_i,\alpha_i)}2 \left( h_{i 1} x^\pm_{i 0}
    + x^\pm_{i 0} h_{i 1} \right).
\end{equation*}
\end{Lemma}

\begin{proof}
  We rewrite the second equality in \lemref{lem:HX} with $i=j$, $r= 1$
  as
\begin{equation}\label{eq:2.7}
  [h_{i 1}, x^\pm_{i 1}] - [h_{i 0}, x^\pm_{i 2}]
  = \pm \frac{(\alpha_i,\alpha_i)}2 \left(
    h_{i 0} x^\pm_{i 1} + x^\pm_{i 1} h_{i 0}
    \right).
\end{equation}

We apply $[\cdot, x^{\mp}_{i0}]$ to \eqref{eq:relexXX10} and combine the resulting relation with \eqref{eq:2.7} and \lemref{lem:XXii} to obtain the desired conclusion.
\end{proof}

\begin{Lemma}\label{lem:XX} Suppose that $i,j\in I$ and $i\neq j$.   The relations \eqref{eq:relXX} and \eqref{eq:relexXX} hold for any
  $r$ and $s$.
\end{Lemma}

\begin{proof}
  We prove \eqref{eq:relexXX} by induction on $r$ and $s$. The same
  argument applies also to \eqref{eq:relXX}. The initial case $r=s=0$
  is our assumption.

  Let $X^\pm(r,s)$ be the result of subtracting the right-hand side of
  \eqref{eq:relexXX} from the left-hand side. We apply $[\tilde h_{i 1}, \cdot]$ and
  $[\tilde h_{j 1},\cdot]$ to \eqref{eq:relexXX} to get
  \begin{equation*}
    \begin{split}
    & 0 = (\alpha_i,\alpha_i) X^\pm(r+1,s) + (\alpha_i,\alpha_j) X^\pm(r,s+1),
\\
    & 0 = (\alpha_i,\alpha_j) X^\pm(r+1,s) + (\alpha_j,\alpha_j) X^\pm(r,s+1).
    \end{split}
  \end{equation*}
Since the determinant of
\(
\begin{pmatrix}
  (\alpha_i,\alpha_i) &  (\alpha_i,\alpha_j) \\
   (\alpha_j,\alpha_i) & (\alpha_j,\alpha_j)
\end{pmatrix}
\) is nonzero by assumption, we have $X^\pm(r+1,s) = 0 = X^\pm(r,s+1)$. Therefore, the assertion is true by induction. 
\end{proof}

\begin{Lemma}\label{lem:HXij}
Suppose that $i,j\in I$. Then \eqref{eq:relexHX} holds for $(r,s)=(1,0)$.
\end{Lemma}

\begin{proof}
The case $i=j$ is provided by \lemref{lem:HXii}, so let's assume that $i\ne j$.  We prove the $+$ case, the $-$ case can be proved in the same way using $h_{i2} = [x_{i0}^+, x_{i2}^-]$, which holds by \lemref{lem:XXii}. (Recall that, by definition, $h_{i2} = [x_{i2}^+, x_{i0}^-]$.) We simply use \lemref{lem:XX}:
  \begin{align*}
    [h_{i2}, x^+_{j0}] = {} & [ [x^+_{i2}, x^-_{i 0}], x^+_{j 0}]
    = [ [x^+_{i2}, x^+_{j0}], x^-_{i 0}]
\\
    = {} & [ [x^+_{i1}, x^+_{j1}], x^-_{i 0}]
    + \frac{(\alpha_i,\alpha_j)}2
    [x^+_{i1} x^+_{j0} + x^+_{j0} x^+_{i1}, x^-_{i 0}]
\\
    = {} & [h_{i1}, x^+_{j1}]
    + \frac{(\alpha_i,\alpha_j)}2 
    \left( h_{i1} x^+_{j0} + x^+_{j0} h_{i1}\right). \qedhere
\end{align*}
\end{proof}

We are now prepared to introduce the element $\tilde h_{i 2}$. For each $i\in I$, we define $\tilde h_{i 2}$ by
\begin{equation}\label{eq:th2}
  \tilde h_{i 2} =  h_{i 2} - h_{i 0} h_{i 1} + \frac13 h_{i 0}^3.
\end{equation}

The next proposition is a special case of Lemma 1.4 in \cite{Lev}. 

\begin{Proposition}\label{prop:hx} For any $i,j\in I$, the following identity holds:
\begin{equation}\label{eq:hx2}
  [\tilde h_{i 2}, x^\pm_{j 0}]
  = \pm (\alpha_i,\alpha_j) x^\pm_{j 2}
  \pm \frac1{12} (\alpha_i,\alpha_j)^3 x^\pm_{j 0}.
\end{equation}
\end{Proposition}
\begin{proof}
This follows from \lemref{lem:HXii} and \lemref{lem:HXij}. Here are the details for the sake of the reader. 
{\allowdisplaybreaks[4]
  \begin{align*}
[\tilde h_{i 2}, x^\pm_{j 0}] = {} & [h_{i 2}, x^\pm_{j 0}]  - [h_{i 0} h_{i 1}, x^\pm_{j 0}] + \frac13 [h_{i 0}^3, x^\pm_{j 0}]
\\
     = {} &  [h_{i 1}, x^\pm_{j 1}] 
     \pm (\alpha_i,\alpha_j) \Biggl[
    \frac12 \left(
       h_{i 1} x^\pm_{j 0} + x^\pm_{j 0} h_{i 1}
       \right)
       - \left(
         x^\pm_{j 0} h_{i 1} + h_{i 0} x^\pm_{j 1}
         \right)
         - \frac12
         \left(h_{i 0}^2 x^\pm_{j 0} + h_{i 0} x^\pm_{j 0} h_{i 0}
           \right)
\\
    &\ + \frac13 \left(
      x^\pm_{j 0} h_{i 0}^2 + h_{i 0} x^\pm_{j 0} h_{i 0}
      + h_{i 0}^2 x^\pm_{j 0}
      \right)\Biggr]
      \qquad \text{by  \lemref{lem:HXij}, \eqref{eq:relHX'} and \eqref{eq:relexHX2'}};
\\
  = {} &
  [h_{i 0}, x^\pm_{j 2}]
     \pm (\alpha_i,\alpha_j) \Biggl[
      \frac12 \left( h_{i 0} x^\pm_{j 1} + x^\pm_{j 1}h_{i 0}
       \right)
       + \frac12 [h_{i 1}, x^\pm_{j 0}] - h_{i 0} x^\pm_{j 1}
       \\
       &\ + \frac16\left(
         2 x^\pm_{j 0} h_{i 0}^2 - h_{i 0} x^\pm_{j 0} h_{i 0}
         - h_{i 0}^2 x^\pm_{j 0}
         \right)\Biggr]
         \qquad \text{by  \lemref{lem:HX}; }
\\
  = {} & \pm (\alpha_i,\alpha_j) \Biggl[
  x^\pm_{j 2} + \frac12 \left([h_{i 1}, x^\pm_{j 0}]
    - [h_{i 0}, x^\pm_{j 1}]\right)
  + \frac16 \left( [x^\pm_{j 0}, h_{i 0}] h_{i 0}
    + [x^\pm_{j 0}, h_{i 0}^2] \right) \Biggr]
\\
  = {} & \pm (\alpha_i,\alpha_j) \Biggl[
  x^\pm_{j 2} \pm \frac{(\alpha_i,\alpha_j)}4
  \left(h_{i 0} x^\pm_{j 0} + x^\pm_{j 0}h_{i 0}\right)
  \mp \frac{(\alpha_i,\alpha_j)}6 \left( 2 x^\pm_{j 0} h_{i 0}
    + h_{i 0} x^\pm_{j 0}\right) \Biggr] \\
    &  \qquad\text{by \eqref{eq:relexHX2'}}
\\
  = {} & \pm (\alpha_i,\alpha_j) \Biggl[
  x^\pm_{j 2} \pm \frac{(\alpha_i,\alpha_j)}{12}
  [h_{i 0}, x^\pm_{j 0}] \Biggr]
  = \pm (\alpha_i,\alpha_j) x^\pm_{j 2}
  \pm \frac{(\alpha_i,\alpha_j)^3}{12} x^\pm_{j 0}. \qedhere
\end{align*}
}
\end{proof}

Now we are ready to check several cases of \eqref{eq:relDS}.

\begin{Lemma}\label{lem:lowDS}
  The relation \eqref{eq:relDS} holds for the following cases:
  \begin{enumerate}
  \renewcommand{\labelenumi}{\textup{(\theenumi)}}%
  \item $r_1 = \cdots = r_b = 0$, $s\in\Z_{\ge 0}$,
  \item $r_1 = 1$, $r_2 = \cdots = r_b = 0$, $s\in\Z_{\ge 0}$,
  \item $r_1 = 2$, $r_2 = \cdots = r_b = 0$, $s=0$,
  \item \textup($b\ge 2$ and\textup) $r_1 = r_2 = 1$, $r_3 =
    \cdots = r_b = 0$, $s=0$. \label{Serre4}
  \end{enumerate}
\end{Lemma}

\begin{proof}
Let $\vec{r} = (r_1,\dots, r_{b})$ and denote the left hand side of
\eqref{eq:relDS} by $X^\pm(\vec{r},s)$. We first show $X^\pm(\vec{0},
s) = 0$ by induction on $s\geq 0$. If $s=0$, this is just \eqref{eq:relDS'}.
Suppose that $X^\pm(\vec{0},s) = 0$ for some $s\geq 0$. We apply $[\tilde h_{i 1},\cdot]$ and
$[\tilde h_{j 1},\cdot]$ to $X^\pm(\vec{0},s)$ to get
\begin{equation*}
  \begin{split}
  & \frac{(\alpha_i,\alpha_i)}{(b-1)!}X^\pm((1,0,\dots,0), s) +
  \frac{(\alpha_i,\alpha_j)}{b!}  X^\pm(\vec{0}, s+1) = 0,
\\
  & \frac{(\alpha_i,\alpha_j)}{(b-1)!} X^\pm((1,0,\dots,0), s) +
  \frac{(\alpha_j,\alpha_j)}{b!} X^\pm(\vec{0}, s+1) = 0.
  \end{split}
\end{equation*}
Since the determinant of
\(
\begin{pmatrix}
  (\alpha_i,\alpha_i) & (\alpha_i,\alpha_j) \\
  (\alpha_j,\alpha_i) & (\alpha_j,\alpha_j)
\end{pmatrix}
\) is nonzero by hypothesis, we obtain that 
\begin{equation*}
X^\pm((1,0,\dots,0), s) = 0 = X^\pm(\vec{0},s+1).
\end{equation*}
Therefore, by induction we have $X^\pm(\vec{0},s)=0$ for all $s\geq 0$.
We simultaneously have proven that $X^\pm((1,0,\dots,0), s) = 0$ for all $s\geq 0$.

Next, we apply $[\tilde h_{i 2},\cdot]$ to $X^\pm(\vec{0},0) = 0$. By
\eqref{eq:hx2} we have
\begin{multline*}
  0 = b (\alpha_i,\alpha_i) X^\pm((2,0,\dots,0),0)
  + \frac{b (\alpha_i,\alpha_i)^3}{12} X^\pm(\vec{0},0)
\\
  + (\alpha_i,\alpha_j) X^\pm(\vec{0},2)
  + \frac{(\alpha_i,\alpha_j)^3}{12} X^\pm(\vec{0},0).
\end{multline*}
Since the last three terms vanish, we have
\(
  X^\pm((2,0,\dots,0),0) = 0.
\)

In order to prove \eqref{Serre4}, we apply $[\tilde h_{i 1}, \cdot]$ to $X^\pm((1,0,\dots,0),0) = 0$.
We have
\begin{multline*}
  0 = 
   \frac{(\alpha_i,\alpha_i)}{(b-1)!} X^\pm((2,0,\dots, 0),0)
\\
  + \frac{(\alpha_i,\alpha_i)}{(b-2)!} X^\pm((1,1,0,\dots, 0),0)
  +  \frac{(\alpha_i,\alpha_j)}{(b-1)!} X^\pm((1,0,\dots,0),1).
\end{multline*}
Since the first and third terms vanish, we have
$X^\pm((1,1,0,\dots, 0),0) = 0$.
\end{proof}

We move on to proving the relation \eqref{eq:relHH} with $i=j$,
$(r,s) = (1,2)$ from \lemref{lem:lowDS}(4): see \propref{prop:HH} and \propref{prop:HH2}. A few intermediary lemmas will be necessary. The argument was
originally noticed in \cite[one paragraph after the proof of Prop.~2.1]{Gu2} for type $\hat\algsl_{n+1}$. Since the proof was omitted there, we reproduce it here.
\begin{Lemma}\label{lem:hxhx} We have
\begin{equation*}
  [h_{j 1}, x^{\pm}_{i 1}] = 
  \frac{(\alpha_i,\alpha_j)}{(\alpha_i,\alpha_i)} 
  [h_{i 1}, x^{\pm}_{i 1}] \pm \frac{(\alpha_i,\alpha_j)}2
  \left(\{h_{j 0}, x^{\pm}_{i 1}\} - \{ h_{i 0}, x^{\pm}_{i 1}\} \right)  \; \text{ for all }\; i,j\in I.
\end{equation*}
\end{Lemma}
\begin{proof}
  The left-hand side is equal to
  \begin{align*}
    \pm \frac1{(\alpha_i,\alpha_i)}
  [h_{j 1}, [h_{i 1}, & x^{\pm}_{i 0} ] ] - \frac12 [h_{j 1}, \{ h_{i 0}, x^{\pm}_{i 0}\} ]
\\
  = {} & \pm \frac1{(\alpha_i,\alpha_i)} [h_{i 1}, [h_{j 1}, x^{\pm}_{i 0}]]
  - \frac{(\alpha_i,\alpha_j)}2 \left( \pm \{h_{i 0}, x^{\pm}_{i 1}\}
    \pm \frac{1}{2} \{ h_{i 0}, \{ x^{\pm}_{i 0}, h_{j 0}\}\} \right)
\\
  = {} & \frac{(\alpha_i,\alpha_j)}{(\alpha_i,\alpha_i)} 
  [h_{i 1}, x^{\pm}_{i 1}  + \frac12 \{ h_{j 0}, x^{\pm}_{i 0}\} ]
  \mp \frac{(\alpha_i,\alpha_j)}2 \left(\{ h_{i 0}, x^{\pm}_{i 1}\} 
    + \frac{1}{2} \{h_{i 0}, \{x^\pm_{i 0}, h_{j 0}\}\} \right).
  \end{align*}
Using
\begin{equation*}
 [h_{i1},\{h_{j0},x_{i0}^\pm\}]=\pm (\alpha_i,\alpha_i)\left\{h_{j0},x_{i1}^\pm+\frac{1}{2}\{h_{i0},x_{i0}^\pm\}\right\} \text{ and } \{h_{j 0}, \{ h_{i 0}, x^\pm_{i 0}\}\} = \{ h_{i 0}, \{ x^\pm_{i 0}, h_{j 0}\}\},
\end{equation*}
we find that this is equal to the right-hand side.
\end{proof}

\begin{Lemma}\label{lem:h2h0}
 For all $i,j\in I$, we have 
  \begin{equation*}
    [h_{i 2}, h_{j 0}] = 0.
  \end{equation*}
\end{Lemma}

\begin{proof}
We have
  \begin{equation*}
    [h_{i 2}, h_{j 0}] = [[x^+_{i 2}, x^-_{i 0}], h_{j 0}]
    = [[x^+_{i 2}, h_{j 0}], x^-_{i 0}] + [x^+_{i 2}, [x^-_{i 0}, h_{j 0}]].
  \end{equation*}
Employing the first relation in \lemref{lem:HX}, we see that this expression is equal to zero.
\end{proof}

\begin{Proposition}\label{prop:HH}
Let $i,j\in I$ be such that $a_{ij} = -1$. Then
\begin{equation*}
  [h_{i 1}, h_{i 2}] = [h_{i 1}, [x^+_{i 1}, x^-_{i 1}]] = 0.
\end{equation*}
\end{Proposition}

\begin{proof}
  For brevity, we suppose $(\alpha_i,\alpha_i) = 2$ and
  $(\alpha_i,\alpha_j) = -1$. 

  The first equality follows from \lemref{lem:XXii}, so we prove the
  second equality.

We start with
\begin{equation*}
  0 = [x^+_{i 1}, [x^+_{i 1}, x^+_{j 0}]],
\end{equation*}
which holds by \lemref{lem:lowDS}(4).
We apply $[\cdot,x^-_{j 1}]$ and use Lemmas \ref{lem:XX} and \ref{lem:XXii} to get
\(
   0 = [x^+_{i 1}, [x^+_{i 1}, h_{j 1}]].
\)
We then apply $[\cdot,x^-_{i 0}]$:
\begin{equation*}
  \begin{split}
  0 &= [h_{i 1}, [x^+_{i 1}, h_{j 1}]] + [x^+_{i 1}, [h_{i 1}, h_{j 1}]]
  + [x^+_{i 1}, [x^+_{i 1}, x^-_{i 1}]]
  + \frac12 [x^+_{i 1}, [x^+_{i 1}, \{h_{j 0}, x^-_{i 0}\}]]
\\
   &= [h_{i 1}, [x^+_{i 1}, h_{j 1}]] 
  + [x^+_{i 1}, [x^+_{i 1}, x^-_{i 1}]]
  + \frac12 [x^+_{i 1}, \{x^+_{i 1}, x^-_{i 0}\} + \{h_{j 0}, h_{i 1}\}].
  \end{split}
\end{equation*}
Here we have used that $[h_{i 1}, h_{j 1}]=0$. We apply $[\cdot,x^-_{i 0}]$ again to obtain
\begin{equation}\label{eq:ss}
  \begin{split}
    0 &= 
    \begin{aligned}[t]
    & -2[ x^-_{i 1}, [x^+_{i 1}, h_{j 1}]]
    - [\{h_{i 0}, x^-_{i 0}\}, [x^+_{i 1}, h_{j 1}]]
    + [h_{i 1}, [x^+_{i 1}, x^-_{i 1}]]
\\
    & 
    + \frac12 [h_{i 1}, [x^+_{i 1}, \{x^-_{i 0},h_{j 0}\}]]
    + [h_{i 1}, [x^+_{i 1}, x^-_{i 1}]] + [x^+_{i 1}, [h_{i 1}, x^-_{i 1}]]
\\
    & 
    -[x^+_{i 1}, [x^+_{i 1}, (x^-_{i 0})^2]
    + \frac12 [[x^+_{i 1}, \{x^+_{i 1}, x^-_{i 0}\} + \{h_{j 0}, h_{i 1}\}],
    x^-_{i 0}]
    \end{aligned}
\\
  &=
  \begin{aligned}[t]
  & -2[ x^-_{i 1}, [x^+_{i 1}, h_{j 1}]]
  + 2[h_{i 1}, [x^+_{i 1}, x^-_{i 1}]]
  + [x^+_{i 1}, [h_{i 1}, x^-_{i 1}]]
\\
  &
  - [\{h_{i 0}, x^-_{i 0}\}, [x^+_{i 1}, h_{j 1}]]
  + [h_{i 1}, \{x^+_{i 1},x^-_{i 0}\}]
  + \frac12 [x^+_{i 1}, \{h_{j 0}, [h_{i 1}, x^-_{i 0}]\}].
  \end{aligned}
  \end{split}
\end{equation}
From \lemref{lem:hxhx} we have
\begin{align*}
  -2[ x^-_{i 1}, [x^+_{i 1}, h_{j 1}]]
 = {} & [x^-_{i 1}, [x^+_{i 1}, h_{i 1}]]
  - [x^-_{i 1}, \{h_{j 0}, x^+_{i 1}\} - \{h_{i 0}, x^+_{i 1}\}]
\\
 = {} & [x^-_{i 1}, [x^+_{i 1}, h_{i 1}]]
  + 3 \{ x^-_{i 1}, x^+_{i 1}\} - \{ h_{j 0}, [x^-_{i 1}, x^+_{i 1}]\}
  + \{ h_{i 0}, [x^-_{i 1}, x^+_{i 1}]\},
\end{align*}
and also 
\begin{align*}
  -[\{h_{i 0}, x^-_{i 0}\}, [x^+_{i 1}, h_{j 1}]] = {} & -\frac12 \left[\{ h_{i 0}, x^-_{i 0}\}, 
  [h_{i 1}, x^+_{i 1}] + \{h_{j 0}, x^+_{i 1}\} - \{ h_{i 0}, x^+_{i 1}\}
  \right]
\\
  = {} &  - \{[h_{i 1}, x^+_{i 1}], x^-_{i 0}\}
  - \frac12 \{h_{i 0}, [x^-_{i 0}, [h_{i 1}, x^+_{i 1}]]\}
  - \{\{h_{j 0}, x^+_{i 1}\}, x^-_{i 0}\}
\\
  & + \frac12 \{h_{i 0}, \{ h_{j 0}, h_{i 1}\} + \{x^-_{i 0}, x^+_{i 1}\}\}
  + \{\{h_{i 0}, x^+_{i 1}\}, x^-_{i 0}\}
\\
  & - 2 h_{i 1} h_{i 0}^2 + \{h_{i 0}, \{x^-_{i 0}, x^+_{i 1}\}\}.
\end{align*}
We substitute these into \eqref{eq:ss} to get
\begin{align}\label{eq:sss}
 - 3 [h_{i 1}, [x^+_{i 1}, x^-_{i 1}]] = {} & 
 3 \{ x^-_{i 1}, x^+_{i 1}\} - \{h_{j 0}, [x^-_{i 1}, x^+_{i 1}]\}
  + \{ h_{i 0}, [x^-_{i 1}, x^+_{i 1}]\} \nonumber 
\\
  & - \{[h_{i 1}, x^+_{i 1}], x^-_{i 0}\}
  - \frac12 \{h_{i 0}, [x^-_{i 0}, [h_{i 1}, x^+_{i 1}]]\}
  - \{\{h_{j 0}, x^+_{i 1}\}, x^-_{i 0}\} \nonumber 
\\
  & + \frac12 \{ h_{i 0}, \{h_{j 0}, h_{i 1}\} + \{x^-_{i 0}, x^+_{i 1}\}\}
  + \{\{h_{i 0}, x^+_{i 1}\}, x^-_{i 0}\}
\\
  & - 2 h_{i 1} h_{i 0}^2 + \{ h_{i 0}, \{x^-_{i 0}, x^+_{i 1}\}\}
  + \{[h_{i 1}, x^+_{i 1}], x^-_{i 0}\} \nonumber 
\\
  & + \{x^+_{i 1}, [h_{i 1}, x^-_{i 0}]\} + \frac12
  \{ h_{j 0}, [x^+_{i 1}, [h_{i 1}, x^-_{i 0}]]\}
  + \frac 12 \{x^+_{i 1}, [h_{i 1}, x^-_{i 0}]\}. \nonumber 
\end{align}
We then substitute
\begin{align*}
  -\frac12 \{h_{i 0}, [x^-_{i 0}, [h_{i 1}, x^+_{i 1}]]\}
  = {} & - \{h_{i 0}, [x^-_{i 1}, x^+_{i 1}]\}
  - \frac12 \{h_{i 0}, [\{x^-_{i 0}, h_{i 0}\}, x^+_{i 1}]\}
\\
  = {} &  - \{h_{i 0}, [x^-_{i 1}, x^+_{i 1}]\}
  + 2 h_{i 1} h_{i 0}^2   - \{h_{i 0}, \{x^-_{i 0}, x^+_{i 1}\}\},
\\
  \frac32 \{x^+_{i 1}, [h_{i 1}, x^-_{i 0}]\}
  = {} & -3 \{x^+_{i 1}, x^-_{i 1}\} - \frac32 \{x^+_{i 1}, \{h_{i 0}, x^-_{i 0}\}\},
\\
  \frac12 \{h_{j 0}, [x^+_{i 1}, [h_{i 1}, x^-_{i 0}]]\}
  = {} & - \{h_{j 0}, [x^+_{i 1}, x^-_{i 1}]\}
  - \frac12 \{h_{j 0}, [x^+_{i 1}, \{h_{i 0}, x^-_{i 0}\}]\}
\\
  = {} & - \{h_{j 0}, [x^+_{i 1}, x^-_{i 1}]\}
  + \{h_{j 0}, \{x^+_{i 1}, x^-_{i, 0}\}\}
  -\frac12 \{h_{j 0}, \{h_{i 0}, h_{i 1}\}\}
\end{align*}
into \eqref{eq:sss} to get
\begin{align*}
  - 3[h_{i 1}, [x^+_{i 1}, x^-_{i 1}]]
= {} & - \{\{h_{j 0}, x^+_{i 1}\}, x^-_{i 0}\}
  + \frac12 \{h_{i 0}, \{x^-_{i 0}, x^+_{i 1}\}\}
  + \{\{h_{i 0}, x^+_{i 1}\}, x^-_{i 0}\}
\\
  & - \frac32 \{x^+_{i 1}, \{h_{i 0}, x^-_{i 0}\}\}
  + \{h_{j 0}, \{x^+_{i 1}, x^-_{i 0}\}\}
\\
= {} &
  [x^+_{i 1}, [x^-_{i 0}, h_{j 0}]]
  + \frac 32 [h_{i 0}, [x^+_{i 1}, x^-_{i 0}] ]
    + \frac12 [x^+_{i 1}, [x^-_{i 0}, h_{i 0}]] = 0.
\end{align*}
This is nothing but the assertion.
\end{proof}

\begin{Proposition}\label{prop:HH2}
  Assume that
$[h_{i 1}, h_{i 2}] = 0$ and $(\alpha_i,\alpha_j)\neq 0$. Then we have
\begin{equation*}
  [h_{j 1}, h_{j 2}] = 0.
\end{equation*}
\end{Proposition}

Together with \propref{prop:HH} this gives \eqref{eq:relHH2} for any
$i$, $j$ because the Dynkin diagram of $\g$ is connected, this being a consequence of the assumption that the Cartan matrix of $\g$ is indecomposable. We are thus able to conclude the proof of \thmref{thm:deduction}.

\begin{proof}
  By the assumption and \lemref{lem:h2h0}, we have $[\tilde h_{i
    1},h_{i 2}] = 0$. Therefore
  \begin{alignat*}{2}
    0 &= [\tilde h_{i 1},h_{i 2}] = [\tilde h_{i 1}, [x^+_{i 1}, x^-_{i 1}]] 
    & \quad & \text{by \lemref{lem:XXii}}
\\
   &= [[\tilde h_{i 1}, x^+_{i 1}], x^-_{i 1}]
   + [x^+_{i 1}, [\tilde h_{i 1}, x^-_{i 1}]] &&
\\
   &= \frac{(\alpha_i,\alpha_i)}{(\alpha_i,\alpha_j)}
   \left( [[\tilde h_{j 1}, x^+_{i 1}], x^-_{i 1}]
     + [x^+_{i 1}, [\tilde h_{j 1}, x^-_{i 1}]]\right)
   &\quad& \text{by \lemref{lem:hxhx}}
\\
  &= \frac{(\alpha_i,\alpha_i)}{(\alpha_i,\alpha_j)} [\tilde h_{j 1}, h_{i 2}].   & \quad & \text{by \lemref{lem:XXii}}
  \end{alignat*}
  We take $\tilde h_{i 2}$ as in \eqref{eq:th2}. Then we have $[\tilde
  h_{j 1},\tilde h_{i 2}] = 0$ and we apply $[\cdot,x^+_{j 0}]$ to this to get:
  \begin{equation*}
    0 = (\alpha_j,\alpha_j) [x^+_{j 1}, \tilde h_{i 2}]
    + (\alpha_i,\alpha_j)[\tilde h_{j 1}, x^+_{j 2}]
    + \frac{(\alpha_i,\alpha_j)^3}{12} [\tilde h_{j 1}, x^+_{j 0}],
  \end{equation*}
where we have used \eqref{eq:hx2}.

We next apply $[\cdot, x^-_{j 0}]$ to this and, using again \eqref{eq:hx2}, we obtain:
\begin{align}
  0 = {}& (\alpha_j,\alpha_j)[h_{j 1}, \tilde h_{i 2}]
      - (\alpha_j,\alpha_j)(\alpha_i,\alpha_j) [x^+_{j 1}, x^-_{j 2}]
      - \frac{(\alpha_j,\alpha_j)(\alpha_i,\alpha_j)^3}{12}
      [x^+_{j 1}, x^-_{j 0}] \nonumber 
\\
     & - (\alpha_i,\alpha_j)(\alpha_j,\alpha_j)[x^-_{j 1}, x^+_{j 2}]
     + (\alpha_i,\alpha_j)[\tilde h_{j 1}, [x^+_{j 2}, x^-_{j 0}]]  \nonumber 
\\
     & - \frac{(\alpha_j,\alpha_j)(\alpha_i,\alpha_j)^3}{12}
     [x^-_{j 1}, x^+_{j 0}] \nonumber 
\\
    = {} & - (\alpha_i,\alpha_j)(\alpha_j,\alpha_j)
    \left([x^+_{j 1}, x^-_{j 2}] + [x^-_{j 1}, x^+_{j 2}]\right)
    + (\alpha_i,\alpha_j)[\tilde h_{j 1}, h_{j 2}]. \label{hj1hj2}
\end{align}
We simplify the last expression as follows. Start with $[x^+_{j 1},
x^-_{j 1}] = h_{j 2}$ from \lemref{lem:XXii} and apply $[\tilde h_{j 1}, \cdot]$ to it to obtain
\begin{equation*}
  (\alpha_j,\alpha_j)\left([x^+_{j 2}, x^-_{j 1}]
    - [x^+_{j 1}, x^-_{j 2}]\right) = [\tilde h_{j 1},h_{j 2}].
\end{equation*}
Therefore the right-hand side of \eqref{hj1hj2} is $2(\alpha_i,\alpha_j)[\tilde h_{j 1},
h_{j 2}]$, so $[\tilde h_{j 1},h_{j 2}]=0$. 
\end{proof}

\section{Operators on modules in the category $\mathscr O$}\label{Sec:O}

\subsection{Category $\mathscr O$}

The category $\mathscr O$ of modules over a finite-dimensional simple Lie algebra has been studied extensively over the past forty years \cite{Hu}. The definition of this category generalizes naturally for all Kac-Moody algebras (see for instance \cite[\S9.1]{Ka-book}).
It is also possible to extend the notion of category $\mathscr O$ to quantum toroidal algebras and affine Yangians: see \cite{He1,GTL2}.

\begin{Definition}
The category $\mathscr O$ of modules over the Yangian $Y(\g)$ consists of all the modules $V$ such that: \begin{enumerate}
\item $V$ is diagonalizable with respect to $\h$.
\item Each $\h$-weight space $V_{\mu}$ is finite-dimensional ($\mu\in\h^*$).
\item There exist $\lambda_1, \ldots, \lambda_k \in \h^*$ such that if $V_{\mu} \neq 0$, then $\lambda_i - \mu \in \sum_{j\in I} \Z_{\ge 0} \alpha_j$ for 
some $1\leq i\leq k$.
\end{enumerate}
\end{Definition}

One consequence of this definition which we will use implicitly is that if $V$ is a module in $\mathscr O$, $\alpha \in \Delta_+$ and $\mu\in\h^*$, then there exists $N\in\Z_{\ge 0}$ such that $V_{\mu+r\alpha}=0$ for all $r\ge N$. Moreover, $V$ is said to be integrable and in the category $\mathscr O$ if, in addition, such an $N$ can be chosen so that $V_{\mu\pm r\alpha}=0$ for all $r\ge N$.

\subsection{Another presentation of the Yangian and operators on category $\mathscr{O}$}\label{S:YopcatO}

When $\g$ is finite-dimensional, Drinfeld \cite{Drinfeld} gave another presentation of $Y(\g)$ as an algebra generated by elements $x$ and $J(x)$ for $x\in\g$ with the defining relations:
\begin{equation}
  \label{eq:another}
\begin{split}
xy-yx = [x,y] \text{ for all } x,y\in\g,  \text{ $J$ is linear in } & x\in\g, \quad
   J([x,y]) = [x, J(y)],
\\
  [J(x), J([y,z])] + [J(z), J([x,y])] + [J(y), J([z,x])]& = \sum_{a,b,c\in\mathtt{A}} ([x,\xi_a],[[y,\xi_b],[z,\xi_c]]) \{ \xi_a,\xi_b,\xi_c \},
\\
  [[J(x), J(y)], [z, J(w)]] + [[J(z), J(w)], [x, J(y)]]
 = & \sum_{a,b,c\in\mathtt{A}} \big(([x,\xi_a],[[y,\xi_b],[[z,w],\xi_c]]) \\
 & + ([z,\xi_a],[[w,\xi_b],[[x,y],\xi_c]]) \big)\{ \xi_a,\xi_b,J(\xi_c) \}
\end{split}
\end{equation}
where $\{ \xi_a \}_{a \in \mathtt{A}}$ is an orthonormal basis of $\g$, $\mathtt{A}$ being a fixed indexing set of size $\dim \g$, and $\{ \xi_a,\xi_b,\xi_c \} = \frac{1}{24} \sum_{\pi\in S_3} \xi_{\pi(a)} \xi_{\pi(b)} \xi_{\pi(c)}$, $S_3$ being the group of permutations of $\{ a,b,c\}$.

The isomorphism between this presentation and the one provided in \defref{def:Yangian} is given by
\begin{equation}\label{eq:J}
  \begin{split}
  x^\pm_i & \mapsto x^\pm_{i0}, \quad h_i \mapsto h_{i0}
\\
  J(h_i) & \mapsto h_{i 1} + v_i, \qquad
  v_i \defeq \frac14 \sum_{\alpha\in\Delta_+}
  (\alpha,\alpha_i) \{x^+_\alpha, x^-_\alpha\}
  - \frac12 h_i^2,
\\
  J(x^\pm_i) &\mapsto x^\pm_{i 1} + w^\pm_i, \quad
  w^\pm_i \defeq \pm \frac14
  \sum_{\alpha\in\Delta_+} \left\{[x_i^\pm,x^\pm_\alpha],x^\mp_\alpha\right\} 
  - \frac14\{x^\pm_i,h_i\},
  \end{split}
\end{equation}
where, for each $\alpha\in \Delta_+$, $x^\pm_\alpha\in \g_{\pm\alpha}$ are nonzero root vectors normalized so that $(x_\alpha^+,x_\alpha^-)=1$ and $x_{\alpha_i}^\pm=x_i^\pm$. 

The right-hand sides of \eqref{eq:another} and \eqref{eq:J} do not make sense unless $\g$ is finite-dimensional. However, we can change the definition of $v_i$ (and thus of $w_i^{\pm}$) so that it gives a well-defined operator on representations in the category $\mathscr O$ as follows.
First observe that 
\begin{equation*}
  \sum_{\alpha\in\Delta_+} (\alpha,\alpha_i) 
  [x^+_\alpha,x^-_\alpha]
  = \frac{1}{2}\omega(h_i)=\frac{1}{2}c_\g h_i, 
\end{equation*}
where $c_\g$ is the eigenvalue of the Casimir element $\omega\in U(\g)$ in the adjoint representation.
Therefore, we have
\begin{equation*}
 v_i = \frac{1}{8}c_\g h_i + \frac12 \sum_{\alpha\in\Delta_+} (\alpha,\alpha_i)
  x^-_\alpha x^+_\alpha - \frac12 h_i^2.
\end{equation*}

Since, for each $\zeta\in \C$, the assignment 
\begin{equation*}
 \tau_\zeta: J(x)\mapsto J(x)+\zeta x, \quad x\mapsto x \quad \forall \; x\in \g
\end{equation*}
determines an automorphism of the Yangian in the presentation \eqref{eq:another} \cite{Drinfeld}, replacing $v_i$ by $v_i+\zeta h_i$ and $w_i^\pm$ by $w_i^\pm +\zeta x_i^\pm$ in \eqref{eq:J} produces another isomorphism between the two presentations. Indeed, this amounts to composing \eqref{eq:J} with $\tau_\zeta$. In particular, taking $\zeta=-\frac{1}{8}c_\g$, we can annihilate the term $\frac{1}{8}c_\g h_i$ appearing in the above expression for $v_i$, which is essential since $c_\g$ does not admit an interpretation when $\g$ is not finite-dimensional. Now assume that $\g$ is an arbitrary symmetrizable Kac-Moody algebra and, for each $\alpha\in \Delta^+$, choose a basis $\{ x_{\alpha}^{(k)} \}$ of $\g_{\alpha}$
and a dual basis $\{ x_{-\alpha}^{(k)}\}$ of $\g_{-\alpha}$ so that $(x_\alpha^{(k)}, x_{-\alpha}^{(l)}) = \delta_{kl}$ and $x_{\alpha_i}^\pm=x_i^\pm$ for all $i\in I$.  Then the formula
\begin{equation}\label{eq:v_i}
v_i = \frac12 \sum_{\alpha\in\Delta_+} (\alpha,\alpha_i)
  \sum_{k=1}^{\dim\g_\alpha}
  x^{(k)}_{-\alpha} x^{(k)}_\alpha - \frac12 h_i^2,
\end{equation}
gives a well-defined operator on representations in the category $\mathscr O$ as $x^{(k)}_\alpha$ kills a given vector if $\mathrm{ht}(\alpha)$ is sufficiently large.  Here, $\mathrm{ht}(\sum_{i\in I}n_i \alpha_i)=\sum_{i\in I}n_i$ for all $(n_i)_{i\in I}\in \Z^{|I|}$.

The  definition of the operators $w^\pm_i$ can then be determined from \eqref{eq:relHX} and \eqref{eq:relexHX} together with the requirement that 
$J([h_i,x^\pm_i]) = [J(h_i), x^\pm_i]$. We obtain
\begin{equation}\label{eq:w_i}
  w^\pm_i = \pm \frac1{(\alpha_i,\alpha_i)} [v_i, x^\pm_i]
  + \frac12 \{h_i,x_i^\pm\}
\end{equation}
which, using \cite[Corollary~2.4]{Ka-book}, can be rewritten as
\begin{equation*} 
  \begin{split}
  w^+_i &=\frac12 \sum_{\alpha\in\Delta_+}
  \sum_{k=1}^{\dim\g_\alpha}
  x^{(k)}_{-\alpha} [x_i^+, x^{(k)}_\alpha] - \frac12 h_i x^+_i,
\\
  w^-_i &=- \frac12 \sum_{\alpha\in\Delta_+}
  \sum_{k=1}^{\dim\g_\alpha}
  [x_i^-, x^{(k)}_{-\alpha}]x^{(k)}_{\alpha}  - \frac12 x^-_i h_i. 
  \end{split}
\end{equation*}
These can also be viewed as well-defined operators on modules in the category $\mathscr O$. Let's see briefly how to obtain $w_i^+$. We have 
\begin{align}
 [v_i,x_i^+] = {} & -\frac{1}{2}(\alpha_i,\alpha_i)(\{h_i,x_i^+\}+h_ix_i^+)\nonumber\\
            &+\frac{1}{2}\sum_{\alpha\in \Delta_+\setminus\{\alpha_i\}}(\alpha,\alpha_i)\sum_{k=1}^{\dim \g_\alpha}(x_{-\alpha}^{(k)}[x_{\alpha}^{(k)},x_i^+]+[x_{-\alpha}^{(k)},x_i^+]x_{\alpha}^{(k)}) \label{nuxi}
\end{align}
and
\begin{align*}
\sum_{\alpha\in \Delta_+\setminus\{\alpha_i\}}(\alpha,\alpha_i) & \sum_{k=1}^{\dim \g_\alpha}(x_{-\alpha}^{(k)}[x_{\alpha}^{(k)},x_i^+]+[x_{-\alpha}^{(k)},x_i^+]x_{\alpha}^{(k)})\\
= {} &\sum_{\beta\in \Delta_+\setminus\{\alpha_i\}}\hspace{-.5em}(\alpha_i-\beta,\alpha_i)\hspace{-.5em}\sum_{k=1}^{\dim \g_{\beta-\alpha_i}}x_{\alpha_i-\beta}^{(k)}[x_i^+,x_{\beta-\alpha_i}^{(k)}] +\sum_{\alpha\in \Delta_+\setminus\{\alpha_i\}}(\alpha,\alpha_i)\sum_{k=1}^{\dim \g_\alpha}[x_{-\alpha}^{(k)},x_i^+]x_{\alpha}^{(k)} \\
& \text{ after setting $\beta = \alpha+\alpha_i$ and using Lemma~1.3 in \cite{Ka-book};} \\
= {} &\sum_{\alpha\in \Delta_+\setminus\{\alpha_i\}} \sum_{k=1}^{\dim \g_{\alpha-\alpha_i}} \left( (\alpha_i-\alpha,\alpha_i)x_{\alpha_i-\alpha}^{(k)}[x_i^+,x_{\alpha-\alpha_i}^{(k)}]
+(\alpha,\alpha_i) x_{\alpha_i-\alpha}^{(k)}[x_i^+,x_{\alpha-\alpha_i}^{(k)}] \right)\\
& \text{ by Corollary~2.4 in \cite{Ka-book};} \\
= {} & (\alpha_i,\alpha_i)\sum_{\beta\in \Delta_+}\sum_{k=1}^{\dim \g_{\beta}}x_{-\beta}^{(k)}[x_i^+,x_{\beta}^{(k)}].
\end{align*}
Combining this with \eqref{nuxi}, we obtain the desired expression for $w_i^+$. 

We set 
\begin{equation} 
J(h_i) = h_{i 1} + v_i \;\text{ and }\; J(x^\pm_i) = x^\pm_{i 1} + w^\pm_i, \label{eq:Jhx} 
\end{equation} 
viewed as operators on modules in $\mathscr{O}$. Later, we will see how to view these also as elements in a completion of the Yangian (\secref{Sec:comp}).

\begin{Remark}
In the special case where $\g$ is of affine type,  the summation which appears in the definition of the operator $v_i$ (see \eqref{eq:v_i}) needs
only to be taken over the set of {\it real\/} positive roots $\Delta^{\re}_+$ because $(\delta,\alpha_i) = 0 \; \forall\, i\in I$ (\cite[(6.2.4)]{Ka-book}). Since the multiplicity of a real root is $1$, we can change the notation $x^{(k)}_{\pm \alpha}$ to $x^\pm_\alpha$ for each $\alpha\in\Delta^\re_+$. The same applies (trivially) to $\g$ of finite type. 
\end{Remark}

\subsection{Commutation relations and reflection operators}\label{S:commrel}

The goal of this subsection is to obtain relations (see \propref{prop:Jroot}) which will be useful in the next section to verify that the coproduct on $Y(\g)$ respects the defining relations of the Yangian. 

In this subsection, we fix a module $V$ in the category $\mathscr{O}$ and view the generators $x_{ir}^{\pm}, h_{ir}$ along with $v_i, w_i^{\pm}$ as operators on $V$.  Let $\rho:Y(\mathfrak{g}) \rightarrow \End_{\C}(V)$ be the corresponding algebra homomorphism.

With the relation \eqref{eq:th} in mind, we set $\tilde v_i = v_i + h_i^2/2$. We will also write $x_i^{\pm}$ for $x_{i0}^{\pm}$ and $h_i$ for $h_{i0}$.

\begin{Lemma}\label{lem:vw}
The following relations hold.
\begin{align}
  & [h_i, v_j] = 0, \quad
  [h_i, w_j^\pm] = \pm(\alpha_i,\alpha_j) w_j^\pm, \label{hiwj}
\\
  & [\tilde v_i, x_j^\pm] = \pm (\alpha_i,\alpha_j) w^\pm_j, \label{vixj}
\\
  & [w^+_i,x^-_j] = \delta_{ij} v_i = [x^+_i, w^-_j], \label{wixj}
\\
  & [w^\pm_i,x^\pm_j] - [x^\pm_i, w^\pm_j]
  = \mp \frac{(\alpha_i,\alpha_j)}2 (x^\pm_i x^\pm_j + x^\pm_j x^\pm_i). \label{wixjpm} 
\end{align}
\end{Lemma}

\begin{proof}
\eqref{hiwj} is straightforward to check, \eqref{vixj} was shown above in the $+$ case when $i=j$ (see \eqref{nuxi}) and the argument when $i\neq j$ is the same.

Let's consider now \eqref{wixj}. By \eqref{vixj} and the first relation in \eqref{hiwj}, we have
\begin{equation*}
 [w_i^+,x_j^-]=\frac{(\alpha_i,\alpha_j)}{(\alpha_i,\alpha_i)}[x_i^+,w_j^-] \quad \forall \; i,j\in I. 
\end{equation*}
Therefore, it suffices to prove the second equality in \eqref{wixj}, namely $[x_i^+,w_j^-]=\delta_{ij}v_i$. Using the definition of $w_j^-$, it is easily verified that
\begin{equation*}
 [x_i^+,w_j^-]=\delta_{ij}v_i-\frac{1}{2}\sum_{\alpha\in \Delta_+\setminus\{\alpha_i\}}\sum_{k=1}^{\dim \g_\alpha}\left(\left[x_j^-,[x_i^+,x_{-\alpha}^{(k)}] \right]x_\alpha^{(k)}+[x_j^-,x_{-\alpha}^{(k)}][x_i^+,x_\alpha^{(k)}] \right).
\end{equation*}
Hence, it suffices to show that the summation on the right-hand side vanishes. This fact follows from Lemma 1.3 and Corollary 2.4 of \cite{Ka-book}. In detail, 
using \cite[Lemma 1.3]{Ka-book} we can write 
\begin{equation*}
 \sum_{\alpha\in \Delta_+\setminus\{\alpha_i\}}\sum_{k=1}^{\dim \g_\alpha}[x_j^-,x_{-\alpha}^{(k)}][x_i^+,x_\alpha^{(k)}]=\sum_{\beta\in \Delta_+\setminus\{\alpha_i\}}\sum_{k=1}^{\dim \g_{\beta-\alpha_i}}[x_j^-,x_{\alpha_i-\beta}^{(k)}][x_i^+,x_{\beta-\alpha_i}^{(k)}].
\end{equation*}
By \cite[Corollary 2.4]{Ka-book},
\begin{equation*}
 \sum_{\beta\in \Delta_+\setminus\{\alpha_i\}}\sum_{k=1}^{\dim \g_{\beta-\alpha_i}}[x_j^-,x_{\alpha_i-\beta}^{(k)}][x_i^+,x_{\beta-\alpha_i}^{(k)}]= -\sum_{\beta\in \Delta_+\setminus\{\alpha_i\}}\sum_{k=1}^{\dim \g_\beta}\left[x_j^-,[x_i^+,x_{-\beta}^{(k)}] \right]x_\beta^{(k)},
\end{equation*}
which proves the assertion. 

Finally, let's establish \eqref{wixjpm}. Again appealing to the definition of $w_k^+$, a straightforward computation shows that \eqref{wixjpm} will hold provided 
\begin{align*}
 \sum_{\alpha\in \Delta_+\setminus\{\alpha_j\}}&\sum_{k=1}^{\dim \g_\alpha}\left([x_{-\alpha}^{(k)},x_j^+][x_i^+,x_\alpha^{(k)}]-x_{-\alpha}^{(k)}\left[x_i^+,[x_j^+,x_\alpha^{(k)}]\right]\right)\\
 &=\sum_{\alpha\in \Delta_+\setminus\{\alpha_i\}}\sum_{k=1}^{\dim \g_\alpha}\left([x_i^+,x_{-\alpha}^{(k)}][x_j^+,x_\alpha^{(k)}]-x_{-\alpha}^{(k)}\left[[x_i^+,x_\alpha^{(k)}],x_j^+\right]\right).
\end{align*}
By Lemma 1.3 and Corollary 2.4 of \cite{Ka-book}, both sides of the above vanish; this is proven in the same way as in the proof of \eqref{wixj}. 
\end{proof}

The previous lemma implies the following equivalences:
\begin{equation}\label{eq:equiv}
\begin{gathered}[m]
 [h_i, J(h_j)] = 0
  \quad \Longleftrightarrow\quad
  [h_{i0}, h_{j1}] = 0, 
\\
  [h_i, J(x_j^\pm)] = J([h_i, x_j^\pm]) 
  \quad \Longleftrightarrow\quad
  [h_{i 0}, x^\pm_{j 1}] = \pm (\alpha_i,\alpha_j) x^\pm_{j 1},
\\
  [J(h_i), x_j^\pm] = J([h_i, x_j^\pm])
  \quad \Longleftrightarrow\quad
  [\tilde h_{i 1}, x^\pm_{j 0}] = \pm (\alpha_i,\alpha_j) x^\pm_{j 1},
\\
  [J(x^+_i), x^-_j] = J([x^+_i, x^-_j]) = [x^+_i, J(x^-_j)]
  \quad \Longleftrightarrow\quad
   [x^+_{i1}, x^-_{j 0}] = \delta_{ij} h_{i 1} = [x^+_{i 0}, x^-_{j 1}],
\\
  [J(x^\pm_i), x^\pm_j] = [x^\pm_i, J(x^\pm_j)]
  \quad \Longleftrightarrow\quad
   [x^\pm_{i1}, x^\pm_{j 0}] - [x^\pm_{i 0}, x^\pm_{j 1}]
   = \pm \frac{(\alpha_i,\alpha_j)}2 \left(
    x^\pm_{i 0} x^\pm_{j 0} + x^\pm_{j 0} x^\pm_{i 0}\right).
\end{gathered}
\end{equation}

If $\alpha$ is a simple root $\alpha_i$, then $J(x_{\alpha}^{\pm})$ has already been defined, and now we want to obtain such operators for any positive real root $\alpha$. By restricting the adjoint action of $\rho(Y(\g))\subset \End_\C(V)$ (viewed as a Lie algebra) to the image of $\g$ under $\rho\circ \iota$ (see \eqref{eqn:iota}), we may equip $\rho(Y(\g))$ with the structure of a $\g$-module. As $x_i^\pm$ operate as the derivations $\ad(x_{i0}^\pm)$ on $\rho(Y(\g))$, which are locally nilpotent because of the Serre relations \eqref{eq:relDS}, Lemma 1.3.5 (b) of \cite{KuBook} implies that the operators 
\begin{equation}\label{tau_i}
  \tau_i \defeq \exp(\ad( e_i)) \exp(-\ad (f_i)) \exp(\ad(e_i)),
\end{equation}
where $e_i = \sqrt{2/(\alpha_i,\alpha_i)}x^+_{i 0}$, $f_i =
\sqrt{2/(\alpha_i,\alpha_i)}x^-_{i 0}$, define algebra automorphisms of $\rho(Y(\g))$. Set 
\begin{equation}\label{v_beta}
 \boldsymbol{v}_\beta=\sum_{\alpha\in \Delta_+}\sum_{k=1}^{\dim \g_\alpha}(\alpha,\beta)x_{-\alpha}^{(k)}x_{\alpha}^{(k)}\in \End_\C(V)\quad \forall \; \beta\in \Delta,
\end{equation}
and let $\widetilde Y(\g,V)$ denote the subalgebra of $\End_\C(V)$ generated by $\rho(Y(\g))$ and $\{\boldsymbol{v}_i\}_{i\in I}$ where $\boldsymbol{v}_i=\boldsymbol{v}_{\alpha_i}$.
\begin{Lemma}\label{L:tau}
 For each $i\in I$, $\tau_i$ extends to an automorphism of $\widetilde Y(\g,V)$ with 
\begin{equation}\label{tau_v}
\tau_i(\boldsymbol{v}_j)=\boldsymbol{v}_{s_i(\alpha_j)}+(\alpha_i,\alpha_j)\{x_{i0}^-,x_{i0}^+\}.
\end{equation}
\end{Lemma}
\begin{proof}
 The lemma will follow from the proof of the formula \eqref{tau_v}, which reduces to a computation in $U(\g)$. Let 
 $\tau_i^{\mathrm{ad}}$ be the automorphism of $U(\g)$ defined exactly as $\tau_i$ but with $e_i$ and $f_i$ viewed as elements of $\g$ rather than operators on $V$ (as in \cite[Lemma 3.8 (b)]{Ka-book}). 
 For $\alpha\in \Delta_+\setminus\{\alpha_i\}$, $\{\tau^\ad_i(x_{\alpha}^{(k)})\}_{k=1}^{\dim \g_\alpha}$ is a basis of $\g_{s_i(\alpha)}$ dual to 
 $\{\tau^\ad_i(x_{-\alpha}^{(k)})\}_{k=1}^{\dim \g_\alpha}$ (a basis of $\g_{-s_i(\alpha)}$) with respect to $(\cdot,\cdot)$. Thus,
 \begin{equation*}
  \tau^\ad_i\left(\sum_{k=1}^{\dim \g_\alpha} x_{-\alpha}^{(k)}x_{\alpha}^{(k)}\right)=\sum_{k=1}^{\dim \g_{s_i(\alpha)}} x_{-s_i(\alpha)}^{(k)}x_{s_i(\alpha)}^{(k)}\quad \forall \; \alpha\in \Delta_+\setminus\{\alpha_i\}.
 \end{equation*}
 The formula \eqref{tau_v} follows from this observation together with the following facts: $\tau^\ad_i(x_i^-x_i^+)=x_i^+x_i^-$, $(\alpha,\alpha_j)=(s_i(\alpha),s_i(\alpha_j))$, $s_i(\Delta_+\setminus\{\alpha_i\})=\Delta_+\setminus\{\alpha_i\}$, and $\rho\circ \iota\circ \tau^\ad_i=\tau_i \circ \iota$.  
\end{proof}
 As a consequence of the lemma, $\tau_i$ can be applied to $J(h_j)$, and thus to $J(x_j^\pm)$ since $J(x_j^\pm)=\pm(\alpha_j,\alpha_j)^{-1}[J(h_j),x_j^\pm]$ for all $j\in I$. This assertion automatically holds when $\g$ is finite-dimensional, but it heavily relies on \lemref{L:tau} when this is not the case.

\begin{Lemma}\label{lem:tauJ}
We have
\begin{gather*}
\tau_i(J(h_j)) = J(h_j) 
  - \frac{2(\alpha_i,\alpha_j)}{(\alpha_i,\alpha_i)} J(h_i).
\end{gather*}
\end{Lemma}
\begin{proof}
  We consider the subalgebra $\algsl_2^{(i)}$ of $Y(\g)$ spanned by $e_i,f_i,h_i$. The space $\End_{\C}(V)$ is a representation of $\algsl_2^{(i)}$ via the adjoint action. Let us prove the lemma first when $j=i$. 
  
  Consider the subspace $\C J(x^+_i) + \C J(h_i)+ \C J(x^-_i)$ of $\End_{\C}(V)$, which is stable under the adjoint action of $\algsl_2^{(i)}$ on $\End_{\C}(V)$ by \lemref{lem:vw} and \eqref{eq:equiv}. There are two cases: either $J(x^+_i) = 0 = J(h_i) = J(x^-_i)$, in which case the lemma is trivial, or that subspace is a three-dimensional irreducible representation of $\algsl_2^{(i)}$. In the latter case, one can check directly that $\tau_i(J(h_i)) = - J(h_i)$.   

Now assume that $j\neq i$.  The operator 
  \begin{equation*}
    J(h_j) - \frac{(\alpha_i,\alpha_j)}{(\alpha_i,\alpha_i)} J(h_i)
  \end{equation*}
  is killed by $\ad(e_i)$ and $\ad(f_i)$ by \eqref{eq:equiv}. Therefore this vector is fixed by
  $\tau_i$, hence the lemma holds also when $j\neq i$.
\end{proof}

Let $\alpha$ be a positive real root. By definition, there is an element $w$ of the Weyl group of $\g$ and a simple root $\alpha_j$ such that $\alpha=w(\alpha_j)$.  Then we define a corresponding (real) root vector by
\begin{equation*}
  x^\pm_\alpha = \tau_{i_1} \tau_{i_2} \cdots \tau_{i_{p-1}}(x^\pm_{i_p}),
\end{equation*}
where $w=s_{i_1}\cdots s_{i_p}$ is a reduced expression of $w$ and $i_p =
j$. (Here, $s_i$ denotes the simple reflection associated to $\alpha_i$.) This is independent of the choice of sequence $i_1$, $i_2$, \dots,
$i_p$ up to a constant multiple. This ambiguity will not be important in the following discussion.

We define
\begin{equation}
  J(x^\pm_\alpha) 
  \defeq \tau_{i_1} \tau_{i_2} \cdots \tau_{i_{p-1}}(J(x^\pm_{i_p})). \label{eq:Jxal}
\end{equation}
It follows from the proposition below that this is also independent of the choice of sequence $i_1$, $i_2$, \dots,
$i_p$ up to a constant multiple.
\begin{Proposition}\label{prop:Jroot}
Suppose $\alpha$ is a positive real root. Then
\begin{equation*}
   [J(h_i), x^\pm_\alpha] = [h_i, J(x^\pm_\alpha)]
   = \pm (\alpha_i, \alpha) J(x^\pm_\alpha) \quad \text{ for all }\; i\in I. 
\end{equation*}
\end{Proposition}
\begin{proof}
 We prove the  proposition by  induction on $p$. If $p = 1$,  then 
  $x^\pm_\alpha = x^\pm_j$, and the assertion is a direct consequence of
  \lemref{lem:vw} and \eqref{eq:equiv}. Suppose the statement of the proposition holds for $x^\pm_\beta$ with $\beta = s_{i_2}\cdots
  s_{i_{p-1}}(\alpha_{i_p})$. Then
  \begin{equation*}
    \begin{split}
    & [J(h_i), x^\pm_\alpha] = \tau_{i_1} \left([ \tau_{i_1}^{-1}J(h_i), x^\pm_\beta]\right)
    = \tau_{i_1} \left([ J(h_i)
      - \frac{2(\alpha_i,\alpha_{i_1})}{(\alpha_{i_1},\alpha_{i_1})}J(h_{i_1}),
      x^\pm_\beta]\right)
\\
    =\; &
    \pm \left( (\alpha_i, \beta) 
      - \frac{2(\alpha_i,\alpha_{i_1})}{(\alpha_{i_1},\alpha_{i_1})}
      (\alpha_{i_1},\beta)\right)
      \tau_{i_1} J(x^\pm_\beta)
      = \pm(s_{i_1}\alpha_i, \beta) J(x^\pm_\alpha)
      = \pm(\alpha_i, \alpha) J(x^\pm_\alpha),
    \end{split}
  \end{equation*}
  where we have used \lemref{lem:tauJ} in the second equality, and the induction assumption in the third. Similarly, 
 the second equality and the relation $[J(h_j), x^\pm_\beta] = [h_j, J(x^\pm_\beta)]$, for all $j\in I$, imply that 
 \begin{equation*}
  [J(h_i), x^\pm_\alpha]=\tau_{i_1}\left([h_i-\frac{2(\alpha_i,\alpha_{i_1})}{(\alpha_{i_1},\alpha_{i_1})}h_{i_1},J(x_\beta^\pm)]\right)
  =\tau_{i_1}\left([\tau_{i_1}^{-1}h_i,J(x_\beta^\pm)]\right)=[h_i,J(x_\alpha^\pm)]. 
 \end{equation*}
Therefore, by induction, the assertion is true for all $\alpha \in \Delta_+^{\mathrm{re}}$.
\end{proof}

\section{Coproduct and modules in the category $\mathscr{O}$}\label{sec:coprod}

\subsection{Casimir operators}
Fix a basis $\{h_{(k)}\}$ of $\h$, and let $\{h^{(k)}\}$ denote its dual basis with respect to the invariant inner product
$(\ ,\ )$.
Given a positive root $\alpha$ we choose a base $\{ x_{\alpha}^{(k)}
\}$ of $\g_{\alpha}$ and the dual base $\{ x_{-\alpha}^{(k)}\}$ of
$\g_{-\alpha}$ so that $(x_\alpha^{(k)}, x_{-\alpha}^{(l)}) =
\delta_{kl}$ as before \eqref{eq:v_i}. 

Let us fix modules $V_1$ and $V_2$ in $\mathscr{O}$. We define an operator $\Omega_+$ on $V_1\otimes V_2$ by:
\begin{equation}
  \Omega_+ \defeq 
  \sum_{k=1}^{\dim \mathfrak h} h^{(k)}\otimes h_{(k)}
  + \sum_{\alpha\in\Delta_+}
  \sum_{k=1}^{\dim \g_\alpha}
  x_{-\alpha}^{(k)}\otimes x_\alpha^{(k)}. \label{HalfCas}
\end{equation}
The definition of $\Omega_+$ is independent of the choice of bases.

Note that $\Omega_+$ does not coincide with the usual Casimir operator when $\g$ is finite-dimensional as it does not contain the term
\(
\sum_{\alpha\in\Delta_+}
  \sum_{k=1}^{\dim \g_\alpha}
  x_\alpha^{(k)} \otimes x_{-\alpha}^{(k)}.
\)
We call it the {\it half Casimir operator}. 

In the general case, 
the Casimir operator $\Omega$ is replaced with the \textit{generalized} Casimir operator (denoted $\Omega^{\mathrm{gen}}$) which is given by
\begin{equation*}
 \Omega^{\mathrm{gen}}=2\nu^{-1}(\rho)+\sum_{k=1}^{\dim \h} h^{(k)} \otimes h_{(k)}+2\sum_{\alpha\in \Delta_+}\sum_{k=1}^{\dim \g_\alpha}x_{-\alpha}^{(k)} \otimes x_{\alpha}^{(k)}.
\end{equation*}
(See \S 2.5 of \cite{Ka-book}. Here, $\nu\colon\h \rightarrow \h^*$ is the linear isomorphism given by $\nu(h_1)(h_2) = (h_1,h_2)$ for all $h_1,h_2\in\h$ and $\rho\in\h^*$ is any linear functional such that $\rho(\alpha_i^{\vee}) = \frac{1}{2} a_{ii}$ for all $i\in I$.)  $\Omega^{\mathrm{gen}}$ coincides with the usual Casimir element when $\g$ is finite-dimensional. 

The half Casimir operator $\Omega_+$ does not commute with coproducts of the generators $x_i^\pm$ or $h_i$. It does however satisfy the following simple commutation relations: 

\begin{Lemma}\label{lem:sq}
  We have
\begin{align}
  & [\square (h), \Omega_+] = 0 \quad \text{for $h\in\mathfrak h$}, \label{eq:hO}
\\
  & [\square (x_i^+), \Omega_+] = - x_i^+ \otimes h_i, \label{eq:xO}
\\
  & [\square (x^-_i), \Omega_+] =  h_i\otimes x^-_i, \label{eq:yO}
\end{align}
for all $i\in I$, where $\square (X) = X\otimes 1 + 1\otimes X$.
\end{Lemma}
\begin{proof}
These relations can be proven using the same techniques as used to prove \lemref{lem:vw}. The first formula is a simple consequence of the definition. The
second and third formulas follow from \cite[Lemmas~1.3, 2.4]{Ka-book}. For example,
\begin{align*}
 [ 1\otimes x^-_i, \Omega_+] = {} &
  \sum_{k=1}^{\dim \mathfrak h} h^{(k)}\otimes [x^-_i, h_{(k)}]
  + \sum_{\alpha\in\Delta_+}
  \sum_{k=1}^{\dim \g_\alpha}x_{-\alpha}^{(k)}\otimes [x^-_i, x_\alpha^{(k)}]
\\
  = {} &
  \sum_{k=1}^{\dim \mathfrak h} h^{(k)}\otimes (h_i, h_{(k)}) x^-_i 
  - x^-_i\otimes h_i + 
  \sum_{\alpha\in\Delta_+ \setminus \{\alpha_i\}}
  \sum_{k=1}^{\dim \g_\alpha}x_{-\alpha}^{(k)}\otimes [x^-_i, x_\alpha^{(k)}]
\\
  = {} &
  h_i \otimes x^-_i - x^-_i\otimes h_i - 
  \sum_{\alpha\in\Delta_+ \setminus \{\alpha_i\}}
  \sum_{k=1}^{\dim \g_{\alpha-\alpha_i}}[x^-_i, x_{-(\alpha-\alpha_i)}^{(k)}]\otimes 
  x_{\alpha-\alpha_i}^{(k)}
\\
  = {} &
  h_i \otimes x^-_i - [x^-_i\otimes 1, \Omega_+]. \qedhere
\end{align*}
\end{proof}

\subsection{The coproduct and statement of the main theorem}\label{Sec:cop}

Let $\square$ be the operator defined by $\square (X) = X\otimes 1 +
1\otimes X$, as in \lemref{lem:sq}. It is {\it not\/} an algebra homomorphism, but satisfies
$\square([X,Y]) = [\square (X), \square (Y)]$ for all $X,Y\in Y(\g)$.

We want to define an algebra homomorphism $\Delta_{V_1,V_2}\colon Y(\g) \rightarrow \End_{\C}(V_1 \otimes V_2)$, so we first specify it on the generators of $Y(\g)$ and then prove afterwards that this assignment does indeed extend to an algebra homomorphism (see \thmref{thm:coproduct}). 
\begin{Definition}\label{def:Delta}
$\Delta_{V_1,V_2}$ assigns to the generators of $Y(\g)$ the following operators in $\End_{\C}(V_1 \otimes V_2)$:  by 
\begin{equation}
  \label{eq:assign}
\begin{gathered}
  \Delta_{V_1,V_2} (h) = \square (h) \quad (\text{for $h\in\mathfrak h$}), \qquad
  \Delta_{V_1,V_2} (x^\pm_{i 0}) = \square (x^\pm_{i 0}),
\\
\begin{aligned}[t]
  \Delta_{V_1,V_2} (h_{i 1}) &= \square (h_{i 1}) + h_{i 0} \otimes h_{i 0} 
  + [h_{i 0}\otimes 1, \Omega_+]\\
  &= h_{i 1}\otimes 1 + 1\otimes h_{i 1} + h_{i 0} \otimes h_{i 0} 
  - \sum_{\alpha\in\Delta_+} (\alpha_i,\alpha)
  \sum_{k=1}^{\dim \g_\alpha} x^{(k)}_{-\alpha}\otimes x^{(k)}_\alpha.
\end{aligned}
\end{gathered}
\end{equation}
\end{Definition}

It follows that \begin{equation}
   \Delta_{V_1,V_2} (\tilde h_{i 1}) = 
   \square (\tilde h_{i 1}) + [h_{i 0}\otimes 1, \Omega_+]. \label{Deltath}
\end{equation}
When $V_1$ and $V_2$ are fixed, we simply write $\Delta$.

\begin{Theorem}\label{thm:coproduct}
  Assume $\g$ is either finite-dimensional (but not $\algsl_2$) or of affine type (but not of type $A_1^{(1)}$ or $A^{(2)}_2$).
  Then the assignment $\Delta$ defines an algebra homomorphism $\Delta
  \colon Y(\g)\to \End_{\C}(V_1 \otimes V_2)$.
\end{Theorem}

\begin{Remark}
When $\g$ is finite-dimensional (including when $\g\cong\algsl_2$), this theorem is already known (see \cite{Drinfeld}) but a proof has never appeared in the literature. In this case, \begin{equation}\label{eq:coproJ}
  \begin{split}
  \Delta (J(h_i)) & = \Delta(\tilde{h}_{i1}) + \Delta(\tilde{v}_i) 
\\  
   &= 
  \square (J(h_i)) 
  + \frac12 \sum_{\alpha\in\Delta_+} (\alpha,\alpha_i)
  \sum_{k=1}^{\dim \g_\alpha}
  \left( x^{(k)}_{\alpha} \otimes x^{(k)}_{-\alpha} 
    - x^{(k)}_{-\alpha}\otimes x^{(k)}_{\alpha} 
     \right)
\\
 &= 
  \square (J(h_i)) 
  +
  \frac 12 [h_{i 0}\otimes 1, \Omega],
  \end{split}
\end{equation}
where $\Omega$ is the Casimir element in $\g\otimes\g$. We have excluded $Y(\algsl_2)$ simply because the proof below would have to be modified in this case. As for the case when $\g$ is of type $A_1^{(1)}$, a formula for a coproduct identical to ours is given in \cite{BoLe}, but it is not clear if their definition of the Yangian of the affine Lie algebra $\widehat{\algsl_2}$ is equivalent to the one which can be found in \cite{BeTs} and in \cite{Ko} (to this effect, see also Remark 5.2 in \cite{Ko}).
\end{Remark}
\begin{Remark}
In \secref{Sec:comp}, we will explain how to replace $\End_{\C}(V_1 \otimes V_2)$ with a completion of the tensor product $Y(\g) \otimes Y(\g)$.
\end{Remark}

The rest of this section is devoted to the proof of this theorem. We will be able to use \thmref{thm:deduction} because we will be working under the same assumptions in the finite or affine setting - see \remref{CMinv}. 
Note also that if we check that the restriction of $\Delta$ to
$Y(\g')$ is an algebra homomorphism, the compatibility
for the extra relations \eqref{eq:Cartan} is straightforward.

Therefore, it is enough to check the compatibility of the relations listed
in \thmref{thm:deduction}.

In what follows, it will be useful to have formulas for $\Delta (x_{i1}^\pm)$ for all $i\in I$. From \eqref{eq:rec} with $r=0$, we obtain:
\begin{equation*}
  \begin{split}
  \Delta (x^\pm_{i 1}) 
  &= \pm (\alpha_i, \alpha_i)^{-1} \Delta ([\tilde h_{i 1}, x^\pm_{i 0}])
\\
  &= \pm (\alpha_i, \alpha_i)^{-1} [\square (\tilde h_{i 1})
  + [h_{i 0}\otimes 1, \Omega_+], \square (x^\pm_{i 0})]
\\
  &= \square (x^\pm_{i 1})
  \pm (\alpha_i, \alpha_i)^{-1}[[
  h_{i 0}\otimes 1, \Omega_+], \square (x^\pm_{i 0})].
  \end{split}
\end{equation*}
We consider the $+$ case first. Note that
\begin{align*}
  [[h_{i 0}\otimes 1, \Omega_+], \square (x^+_{i 0})]
  = \;&- [[ 1\otimes h_{i 0}, \Omega_+], \square (x^+_{i 0})]
\\
  =\; &  - [[1 \otimes h_{i 0}, \square (x^+_{i 0})], \Omega_+]
  - [1 \otimes h_{i 0}, [\Omega_+, \square (x^+_{i 0})]]
\\
  =\; & - (\alpha_i,\alpha_i) [1 \otimes x^+_{i 0}, \Omega_+]
  - [1\otimes h_{i 0}, x^+_{i 0}\otimes h_{i 0}]
  = - (\alpha_i,\alpha_i) [1 \otimes x^+_{i 0}, \Omega_+],
\end{align*}
where we have used \eqref{eq:hO} in the first equality and
\eqref{eq:xO} in the third. Therefore we have
\begin{equation}
   \Delta (x^+_{i 1}) = \square (x^+_{i 1}) - [1 \otimes x^+_{i 0}, \Omega_+]. \label{eq:DeltaX+}
\end{equation}
More explicitly 
\begin{equation*}
  \Delta (x^+_{i 1}) =
    x^+_{i 1}\otimes 1 + 1 \otimes x^+_{i 1}
    + h_{i 0} \otimes x^+_{i 0} - \sum_{\alpha\in \Delta_+}
    \sum_{k=1}^{\dim \g_\alpha} x_{-\alpha}^{(k)}\otimes [x^+_{i 0},x_{\alpha}^{(k)}].
\end{equation*}

Similarly we have
\begin{equation*}
  \begin{split}
   \Delta (x^-_{i 1}) &= \square (x^-_{i 1}) + [x^-_{i 0}\otimes 1, \Omega_+]
\\
   &= x^-_{i 1}\otimes 1 + 1\otimes x^-_{i 1}
   + x^-_{i 0} \otimes h_{i 0} + \sum_{\alpha\in\Delta_+}
   \sum_{k=1}^{\dim \g_\alpha} [x^-_{i 0},x_{-\alpha}^{(k)}]\otimes x_{\alpha}^{(k)}.
  \end{split}
\end{equation*}

\subsection{Proof of \thmref{thm:coproduct}, Part I}

We begin by checking that all the defining relations \eqref{eq:relHH'}-\eqref{eq:relDS'} except \eqref{eq:relHH'} when $r=1=s$ are compatible with $\Delta$. Computations for these relations work for any Kac-Moody algebra $\g$ satisfying the assumptions of \thmref{thm:deduction}.  We do not need to check the relations involving only $h_{i 0}$ or $x^\pm_{i 0}$. 

We first check \eqref{eq:relHH'} with $(r,s) = (0,1)$:
\begin{equation*}
  [\Delta(h_{i 0}), \Delta (\tilde h_{j 1})]
  = [\square(h_{i 0}), \square (\tilde h_{j 1}) + [h_{j 0}\otimes 1,\Omega_+] ]
  = [[\square (h_{i 0}), h_{j0}\otimes 1],\Omega_+]
  + [h_{j 0}\otimes 1, [\square (h_{i 0}), \Omega_+]].
\end{equation*}
This vanishes thanks to \eqref{eq:hO}.

As for \eqref{eq:relHX'}, we must verify that it is preserved when $i\neq j$ and $s=1$. We have
\begin{align*}
 [\Delta(h_{i0}),\Delta(x_{j1}^+)]= {} &[\square(h_{i0}),\square (x^+_{j1}) - [1 \otimes x^+_{j 0}, \Omega_+]] \quad \text{ by }\; \eqref{eq:DeltaX+}\\
                                  = {} &(\alpha_i,\alpha_j)\square(x_{j1}^+)-(\alpha_i,\alpha_j)[1\otimes x_{j0}^+,\Omega_+]-[1\otimes x_{j0}^+,[\square(h_{i0}),\Omega_+]] \quad \text{ by }\; \eqref{eq:relHX'}\\
                                  = {} &(\alpha_i,\alpha_j)\Delta(x^+_{j1}) \quad \text{ by }\; \eqref{eq:hO}\; \text{ and }\;\eqref{eq:DeltaX+}.
\end{align*}
The $\pm=-$ case is verified in the same way. 

Next let us check \eqref{eq:relXX'} with $(r,s) = (1,0)$:
\begin{align*}
[\Delta (x^+_{i 1}), \Delta (x^-_{j 0})]
  = {} & [\square (x^+_{i 1}) - [1\otimes x^+_{i 0}, \Omega_+], \square (x^-_{j 0})] \text{ by \eqref{eq:DeltaX+}}
\\
  = {} & \square ([x^+_{i 1}, x^-_{j 0}])
  - [[1\otimes x^+_{i 0}, \square (x^-_{j 0})], \Omega_+]
  - [1\otimes x^+_{i 0}, [\Omega_+, \square (x^-_{j 0})]]
\\
  = {} & \delta_{ij} \square (h_{i 1})
  - \delta_{ij} [1 \otimes h_{i 0},\Omega_+]
  + [1\otimes x^+_{i 0}, h_{j 0}\otimes x^-_{j 0}]
\\
  = {} & \delta_{ij} \left(\square (h_{i 1})
  + [h_{i 0}\otimes 1, \Omega_+]
  + h_{i 0} \otimes h_{i 0}\right) = \delta_{ij} \Delta (h_{i 1}),
\end{align*}
where we have used \eqref{eq:yO} in the third equality and
\eqref{eq:hO} in the fourth.

The relation \eqref{eq:relXX'} with $(r,s) = (0,1)$ can be checked in
a similar way.

Next we check \eqref{eq:relexHX2'}:
\begin{align*}
[\Delta (\tilde h_{i 1}), \Delta (x^\pm_{j 0})]
  = {} & [\square (\tilde h_{i 1}) + [h_{i 0}\otimes 1, \Omega_+],
  \square (x^\pm_{j 0})] \text{ by \eqref{Deltath}}
\\
  = {} &  \square ([\tilde h_{i 1}, x^\pm_{j 0}]) + 
  [[h_{i 0}\otimes 1, \square (x^\pm_{j 0})], \Omega_+]
  + [h_{i 0}\otimes 1, [\Omega_+, \square (x^\pm_{j 0})]]
\\
  = {} & \pm(\alpha_i, \alpha_j)\left( \square (x^\pm_{j 1})
  + [x^\pm_{j 0}\otimes 1, \Omega_+] \right)
  + [h_{i 0}\otimes 1, [\Omega_+, \square (x^\pm_{j 0})]].
\end{align*}
In the $+$ case, the above is  equal to
\begin{align*}
 (\alpha_i, \alpha_j) & \left( \square (x^+_{j 1})
  + [x^+_{j 0}\otimes 1, \Omega_+] \right)
  + [h_{i 0}\otimes 1, x_{j0}^+\otimes h_{j0}]
\\
= \; &
   (\alpha_i, \alpha_j)\left( \square (x^+_{j 1})
  + [x^+_{j 0}\otimes 1, \Omega_+] 
  + x_{j0}^+\otimes h_{j0}
   \right)
   = (\alpha_i, \alpha_j)(\square (x^+_{j 1}) - [1\otimes x^+_{j 0},\Omega_+]),
\end{align*}
thanks to \eqref{eq:xO}. By \eqref{eq:DeltaX+}, this is precisely $(\alpha_i,\alpha_j)\Delta(x^+_{j 1})$. The $-$ case can be proved in a
similar way. Thus, $\Delta$ preserves the relation \eqref{eq:relexHX2'}.

Let us check that \eqref{eq:relexXX'} is compatible with $\Delta$. We
have
\begin{align*}
[\Delta (x^+_{i 1}), \Delta (x^+_{j 0})]
    =\; & [\square (x^+_{i 1}) - [1\otimes x^+_{i 0}, \Omega_+], \square (x^+_{j 0})] \text{ by }\eqref{eq:DeltaX+}
\\
    =\; & \square ([x^+_{i 1}, x^+_{j 0}]) 
    - [[1\otimes x^+_{i 0}, \square (x^+_{j 0})], \Omega_+]
    - [1\otimes x^+_{i 0}, [\Omega_+,\square (x^+_{j 0})]]
\\
    =\; & \square ([x^+_{i 1}, x^+_{j 0}]) 
    - [1\otimes [x^+_{i 0}, x^+_{j 0}], \Omega_+]
    - [1\otimes x^+_{i 0}, x^+_{j 0} \otimes h_{j0}] \text{ by } \eqref{eq:xO} 
\\
    =\; & \square ([x^+_{i 1}, x^+_{j 0}])
    - [1\otimes [x^+_{i 0}, x^+_{j 0}], \Omega_+]
    + (\alpha_i,\alpha_j) x^+_{j 0} \otimes x^+_{i 0}.
\end{align*}
Exchanging  $i$ and $j$, we also obtain an expression for $[\Delta (x_{j1}^+),\Delta (x_{i0}^+)]$. Adding 
this to the above expression for $[\Delta (x^+_{i 1}), \Delta (x^+_{j 0})]$ yields 
\begin{equation}
   [\Delta (x^+_{i 1}), \Delta (x^+_{j 0})] + [\Delta (x^+_{j 1}), \Delta (x^+_{i 0})]
   = \square \left( [x^+_{i 1}, x^+_{j 0}]  - [x^+_{i 0}, x^+_{j 1}] \right)
   + (\alpha_i, \alpha_j)\left( x^+_{j 0} \otimes x^+_{i 0} + 
     x^+_{i 0} \otimes x^+_{j 0}\right). \label{xi1xj1}
\end{equation}
On the other hand, applying $\Delta$ to the right-hand side of
\eqref{eq:relexXX'}, we have
\begin{equation*}
  \frac{(\alpha_i,\alpha_j)}2 (\square (x^+_{i 0}) \square (x^+_{j 0}) + 
  \square (x^+_{j 0}) \square (x^+_{i 0}))
  = \frac{(\alpha_i,\alpha_j)}2 \left(
    \square \{ x^+_{i 0}, x^+_{j 0} \} + 2\left( x^+_{j 0} \otimes x^+_{i 0} + 
     x^+_{i 0} \otimes x^+_{j 0}\right)\right).
\end{equation*}
This is equal to \eqref{xi1xj1} thanks to \eqref{eq:relexXX'}. This proves the compatibility of $\Delta$ with \eqref{eq:relexXX'} when $\pm=+$.  The same proof works for the $-$ case.

\subsection{Proof of \thmref{thm:coproduct}, Part II}\label{sec:proofII}

It remains to verify that $\Delta$ preserves the relation $[\tilde h_{i1}, \tilde h_{j1}]=0$ for all $i,j\in I$. To accomplish this,  we will need to make use of the assumption that $\g$ is of finite or affine type. Since $\Delta(\tilde h_{k1})=\square(\tilde h_{k1})+[h_{k0}\otimes 1,\Omega_+]$ and $\square([\tilde h_{i1},\tilde h_{j1}])=0$, it suffices to show that 
\begin{equation}
 [[h_{i0}\otimes 1,\Omega_+],[h_{j0}\otimes 1,\Omega_+]]
 =[\square(\tilde h_{j1}),[h_{i0}\otimes1,\Omega_+]]-[\square(\tilde h_{i1}),[h_{j0}\otimes1,\Omega_+]]. \label{suff}
\end{equation}
The left-hand side is the sum over $k\in \mathbb{Z}_{>0}$ of 
\begin{align}
\begin{split}\label{LHS}
 \sum_{\mathrm{ht}(\alpha+\beta)=k}&(\alpha_i,\alpha)(\alpha_j,\beta)[x_\alpha^-\otimes x_\alpha^+,x_\beta^-\otimes x_\beta^+]\\
 &=\frac{1}{2}\sum_{\mathrm{ht}(\alpha+\beta)=k}(\alpha_i,\alpha)(\alpha_j,\beta)\left(\{x_\alpha^-,x_\beta^-\}\otimes [x_\alpha^+,x_\beta^+]+ [x_\alpha^-,x_\beta^-]\otimes \{x_\alpha^+,x_\beta^+\}\right),
 \end{split}
\end{align}
where the sum $\sum_{\mathrm{ht}(\alpha+\beta)=k}$ is taken over all $\alpha,\beta\in \Delta_+^{\mathrm{re}}$ such that $\mathrm{ht}(\alpha+\beta)=k$. The right-hand side of \eqref{suff} is the sum over $k\in \mathbb{Z}_{>0}$ of
 \begin{equation}\label{RHS}
  \sum_{\alpha\in \Delta_+^\mathrm{re}(k)}(\alpha,\alpha_j)[\square(\tilde h_{i1}),x_\alpha^-\otimes x_\alpha^+]-\sum_{\alpha\in \Delta_+^\mathrm{re}(k)}(\alpha,\alpha_i)[\square(\tilde h_{j1}),x_\alpha^-\otimes x_\alpha^+], 
 \end{equation}
 where $\Delta_+^{\mathrm{re}}(k)=\{\alpha\in \Delta_+^{\mathrm{re}}\,:\,\mathrm{ht}(\alpha)=k\}$. 
 Therefore, \eqref{suff} will hold if the following equality is established for all $k\in \Z_{>0}$:
 \begin{equation}
 \label{suff:2}
  \sum_{\mathrm{ht}(\alpha+\beta)=k}(\alpha_i,\alpha)(\alpha_j,\beta)\{x_\alpha^\mp,x_\beta^\mp\}\otimes [x_\alpha^\pm,x_\beta^\pm]
  =2\hspace{-.5em}\sum_{\alpha\in \Delta_+^{\mathrm{re}}(k)}[h_{ij}(\alpha),x_\alpha^\mp]\otimes x_\alpha^\pm,
 \end{equation}
where $h_{ij}(\alpha)=(\alpha,\alpha_j)\tilde h_{i1}-(\alpha,\alpha_i)\tilde h_{j1}$ for all $i,j\in I$ and $\alpha\in \Delta_+^{\mathrm{re}}$.

Since $J(h_k)=\tilde h_{k1}+\tilde v_{k}$, \propref{prop:Jroot} implies that 
\begin{equation}\label{hij=vij}
 [h_{ij}(\alpha),x_\alpha^\mp]
 =[(\alpha,\alpha_j)J(h_i)-(\alpha,\alpha_i)J(h_j),x_\alpha^\mp]-[v_{ij}(\alpha),x_\alpha^\mp]=[x_\alpha^\mp,v_{ij}(\alpha)],
\end{equation}
where $v_{ij}(\alpha)=(\alpha,\alpha_j)\tilde v_{i}-(\alpha,\alpha_i)\tilde v_{j}$.
We claim that 
\begin{equation}\label{Lt}
 [v_{ij}(\alpha),x_\alpha^\mp]= -\frac{1}{2}\sum_{\substack{\beta,\gamma\in \Delta_+^{\mathrm{re}}\\ \beta+\gamma=\alpha}}A_{ij,\beta,\gamma}(x_\alpha^\mp,[x_\beta^\pm,x_\gamma^\pm])x_\beta^\mp x_\gamma^\mp,
 \end{equation}
where $A_{ij,\beta,\gamma}=(\alpha_i,\beta)(\alpha_j,\gamma)-(\alpha_j,\beta)(\alpha_i,\gamma)$ for each $\beta,\gamma\in \Delta_+^{\mathrm{re}}$. We will prove \eqref{Lt} in the case where the symbols $\pm$ and $\mp$ take their upper values $+$ and $-$, respectively. The $(\pm,\mp)=(-,+)$ case follows from an essentially identical argument.

By definition of $\tilde v_i$, we have
\begin{equation*}
 [\tilde v_i,x_\alpha^-]=-\frac{1}{4}h^\vee(\alpha_i,\alpha)x_\alpha^-+\frac{1}{2}\sum_{\beta\in \Delta_+^{\mathrm{re}}}(\beta,\alpha_i)([x_\beta^-,x_\alpha^-]x_\beta^+ + x_\beta^-[x_\beta^+,x_\alpha^-]),
\end{equation*}
which implies the equality
\begin{equation}\label{Lt0}
 [v_{ij}(\alpha),x_\alpha^-]=\frac{1}{2}\sum_{\beta\in \Delta_+^{\mathrm{re}}}A_{ij,\beta,\alpha}([x_\beta^-,x_\alpha^-]x_\beta^+ + x_\beta^-[x_\beta^+,x_\alpha^-]).
\end{equation}
Consider $x_\beta^-[x_\beta^+,x_\alpha^-]$. Assume first that $\gamma=\beta-\alpha$ is positive. If $\gamma$ is an imaginary root, the coefficient $A_{ij,\beta,\alpha}$ vanishes. Therefore, we may assume $\gamma\in \Delta_+^{\mathrm{re}}$. Then
\begin{equation*}
 x_\beta^-[x_\beta^+,x_\alpha^-]=(x_\gamma^-,[x_\beta^+,x_\alpha^-])x_\beta^-x_\gamma^+.
\end{equation*}
If $\beta-\alpha$ is negative, we set instead $\gamma=\alpha-\beta$ (which we may again assume to be a real root) and deduce that 
\begin{equation*}
 x_\beta^-[x_\beta^+,x_\alpha^-]=(x_\gamma^+,[x_\beta^+,x_\alpha^-])x_\beta^-x_\gamma^{-}=-(x_\alpha^-,[x_\beta^+,x_\gamma^+])x_\beta^-x_\gamma^{-}.
\end{equation*}
We thus have 
\begin{equation}
 \label{Lt1}
\sum_{\beta\in \Delta_+^{\mathrm{re}}}A_{ij,\beta,\alpha}x_\beta^-[x_\beta^+,x_\alpha^-]
=\sum_{\substack{\beta,\gamma\in \Delta_+^{\mathrm{re}}\\\gamma+\alpha=\beta }}A_{ij,\gamma,\beta}(x_\gamma^-,[x_\beta^+,x_\alpha^-])x_\beta^-x_\gamma^+
-\sum_{\substack{\beta,\gamma\in \Delta_+^{\mathrm{re}}\\ \beta+\gamma=\alpha}}A_{ij,\beta,\gamma}(x_\alpha^-,[x_\beta^+,x_\gamma^+])x_\beta^-x_\gamma^{-},
\end{equation}
where we have used that $A_{ij,\beta,\beta-\gamma}=A_{ij,\gamma,\beta}$ and $A_{ij,\beta,\beta+\gamma}=A_{ij,\beta,\gamma}$. By similar reasoning, 
\begin{equation*}
 \sum_{\beta\in \Delta_+^{\mathrm{re}}}A_{ij,\beta,\alpha}[x_\beta^-,x_\alpha^-]x_\beta^+=\sum_{\substack{\beta,\gamma\in \Delta_+^{\mathrm{re}}\\\beta+\alpha=\gamma}}A_{ij,\beta,\gamma}(x_\gamma^+,[x_\beta^-,x_\alpha^-])x_\gamma^-x_\beta^+
 =-\sum_{\substack{\beta,\gamma\in \Delta_+^{\mathrm{re}}\\\beta+\alpha=\gamma}}A_{ij,\gamma,\beta}(x_\gamma^-,[x_\beta^+,x_\alpha^-])x_\beta^-x_\gamma^+,
\end{equation*}
which cancels with the first term on the right-hand side of \eqref{Lt1}. This proves that the formula \eqref{Lt} holds when $(\pm,\mp)=(+,-)$.

Combining \eqref{hij=vij} with \eqref{Lt}, we obtain   
\begin{equation*}
 2[h_{ij}(\alpha),x_\alpha^\mp]\otimes x_\alpha^\pm
 =\sum_{\substack{\beta,\gamma\in \Delta_+^{\mathrm{re}}\\ \beta+\gamma=\alpha}}A_{ij,\beta,\gamma}(x_\alpha^\mp,[x_\beta^\pm,x_\gamma^\pm])x_\beta^\mp x_\gamma^\mp \otimes x_\alpha^\pm
 =\sum_{\substack{\beta,\gamma\in \Delta_+^{\mathrm{re}}\\ \beta+\gamma=\alpha}}A_{ij,\beta,\gamma}x_\beta^\mp x_\gamma^\mp \otimes [x_\beta^\pm,x_\gamma^\pm].
\end{equation*}
After adding the last expression to itself with $\beta$ and $\gamma$ exchanged, dividing by two, and then summing over $\alpha\in \Delta_+^{\mathrm{re}}(k)$, we find the following expression for the right-hand side of \eqref{suff:2}:
\begin{equation}\label{R->L}
 2\hspace{-.5em}\sum_{\alpha\in \Delta_+^\mathrm{re}(k)}[h_{ij}(\alpha),x_\alpha^\mp]\otimes x_\alpha^\pm=\frac{1}{2}\sum_{\alpha\in \Delta_+^\mathrm{re}(k)}\sum_{\substack{\beta,\gamma\in \Delta_+^{\mathrm{re}}\\ \beta+\gamma=\alpha}}A_{ij,\beta,\gamma}\{x_\beta^\mp,x_\gamma^\mp\} \otimes [x_\beta^\pm,x_\gamma^\pm].
\end{equation}
Conversely, adding the left-hand side of \eqref{suff:2} to itself with $\alpha$ and $\beta$ exchanged and dividing by two, we deduce that it is equal to
\begin{equation*}
\frac{1}{2}\sum_{\mathrm{ht}(\alpha+\beta)=k}A_{ij,\alpha,\beta}\{x_\alpha^\mp,x_\beta^\mp\}\otimes [x_\alpha^\pm,x_\beta^\pm].
\end{equation*}
Since $A_{ij,\alpha,\beta}=0$ for $\alpha+\beta\in\Delta_+^{\mathrm{im}}$, this coincides with the right-hand side of \eqref{R->L}. This completes the proof of the identity \eqref{suff:2}, and hence the proof of \thmref{thm:coproduct}.

\subsection{Coassociativity}
It follows from \thmref{thm:coproduct} that we can turn the tensor product $V_1 \otimes V_2$ of two representations in the category $\mathscr{O}$ into a representation of $Y(\g)$ which is also in the category $\mathscr{O}$. Indeed, a $Y(\g)$-module $V$ is in $\mathscr{O}$ if $\iota^*(V)$ is in the category $\mathscr{O}$ for $\g$; moreover, $\iota^*(V_1 \otimes V_2) = \iota^*(V_1) \otimes \iota^*(V_2)$. It is very desirable for this coproduct to be compatible with the associativity of the tensor product. 

\begin{Proposition}\label{Prop:mon}
Let $V_1,V_2$ and $V_3$ be $Y(\g)$-modules in the category $\mathscr O$. Then the natural isomorphism of vector spaces 
\begin{equation*}
a_{V_1,V_2,V_3}\colon (V_1\otimes V_2)\otimes V_3 \to V_1\otimes (V_2\otimes V_3)
 \end{equation*}
is an isomorphism of $Y(\g)$-modules. 
\end{Proposition}
\begin{proof}
We need to show that, after identifying the spaces $\End_\C((V_1\otimes V_2)\otimes V_3)$ and $\End_\C(V_1\otimes (V_2\otimes V_3))$ (via $a_{V_1,V_2,V_3}$), we have 
 $\Delta_{V_1\otimes V_2,V_3}=\Delta_{V_1,V_2\otimes V_3}$. Since $Y(\g)$ is generated by $\g$ and $\tilde h_{i1}$ (for all $i\in I$), we need only to establish this equality when both sides are applied to $\tilde h_{i1}$. By \eqref{Deltath}, we have 
\begin{align*}  
   \Delta_{V_1\otimes V_2,V_3}(\tilde h_{i1})=&(\tilde h_{i1}\otimes 1)\otimes 1 +(1\otimes 1)\otimes \tilde h_{i1} +(1\otimes \tilde h_{i1})\otimes 1 \\
                                       &-\sum_{\alpha\in \Delta_+^\re}(\alpha_i,\alpha)\left((x_\alpha^-\otimes x_\alpha^+)\otimes 1 +(x_\alpha^-\otimes 1 +1\otimes x_\alpha^-)\otimes x_\alpha^+ \right),\\
  \Delta_{V_1,V_2\otimes V_3}(\tilde h_{i1})=& \tilde h_{i1}\otimes (1\otimes 1)+ 1\otimes (\tilde h_{i1}\otimes 1)+   1\otimes (1\otimes \tilde h_{i1})\\
                                       &-\sum_{\alpha\in \Delta_+^\re}(\alpha_i,\alpha)\left(1\otimes (x_\alpha^-\otimes x_\alpha^+)+x_\alpha^-\otimes (x_\alpha^+\otimes 1+1\otimes x_\alpha^+) \right).
 \end{align*}
Hence, $\Delta_{V_1\otimes V_2,V_3}(\tilde h_{i1})=\Delta_{V_1,V_2\otimes V_3}(\tilde h_{i1})$ and consequently $a_{V_1,V_2,V_3}$ is an isomorphism of $Y(\g)$-modules. \end{proof}

\section{Coproduct and completions of Yangians}\label{Sec:comp}

The collection of algebra homomorphisms $\Delta_{V_1,V_2}$, which are defined on generators by \eqref{eq:assign}, can be viewed together as a sort of comultiplication on $Y(\g)$ which is coassociative in the sense of \propref{Prop:mon}. Our present goal is to improve on this by showing that each homomorphism $\Delta_{V_1,V_2}\colon Y(\g)\to \End_\C(V_1\otimes V_2)$ can be recovered from a single homomorphism $\Delta\colon Y(\g)\to Y(\g)\widehat\otimes Y(\g)$, where $Y(\g)\widehat\otimes Y(\g)$ is a suitable completion of $Y(\g)\otimes Y(\g)$.

Our first step is to define a completion $\widehat Y(\g)$ of $Y(\g)$ which behaves nicely with respect to modules in the category $\mathscr O$, and from which the definition of $Y(\g)\widehat\otimes Y(\g)$ can be obtained as a special case. 
\subsection{The completion $\widehat Y(\g)$}
Let $\g$ be a symmetrizable Kac-Moody algebra as in \secref{sec:YangianKM}, except that we no longer require the Cartan matrix $(a_{ij})_{i,j\in I}$ to be indecomposable (because we want to consider the Yangian of $\g \oplus \g$). Note that in this case we can still define the Yangian $Y_\hbar(\g)$, and thus $Y(\g)$, using \defref{def:Yangian}. For the purpose of introducing the completion $\widehat Y(\g)$ we need to impose two mild conditions on $Y(\g)$:
\begin{enumerate}[(A)]
\item \label{as.1} We suppose that $Y(\g)$ admits the multiplicative triangular decomposition 
              \begin{equation*}
               Y(\g)\cong Y^- \otimes Y^0 \otimes Y^+,
              \end{equation*}
 where $Y^\pm$ (resp. $Y^0$) denotes the subalgebra of $Y(\g)$ generated $x_{ir}^\pm$ (resp. $h_{ir}$ and $h\in \h$) with $i\in I$ and $r\ge 0$. 
 \item  We also assume that $Y^\pm$ is isomorphic to the quotient of the free algebra on the generators $x_{i,r}^{\pm}$ for all $i\in I, \, r\ge 0$ by 
the ideal corresponding to the relations \eqref{eq:relexXX} and \eqref{eq:relDS}. \label{as.2}
\end{enumerate}
It is very plausible that these assumptions on $Y(\g)$ are always satisfied (even when $\g$ is not affine). Indeed, for affinizations of quantum Kac-Moody algebras
such a result was obtained by D. Hernandez in \cite[Theorem 3.2]{He1}, and the corresponding result for Yangians could most likely be proven using 
exactly the same technique.

Set $\deg x_{ir}^+=1$ for all $i\in I$ and $r\geq 0$. The assumption \eqref{as.2} implies that we have $Y^+=\bigoplus_{k=0}^\infty Y^+[k]$, where $Y^+[k]$ is the span of all monomials of degree $k$ in $Y^+$, and this grading is compatible with the algebra structure on $Y^+$. This together with the assumption \eqref{as.1} imply that we have the vector space grading
\begin{equation*}
 Y(\g)= \bigoplus_{k=0}^\infty Y(\g)[k], \quad \text{where } Y(\g)[k] = Y^{\leq 0}\otimes Y^+[k]
\end{equation*}
and $Y^{\le 0}$ is the subalgebra of $Y(\g)$ generated by $x_{ir}^-$ and $h_{ir}$ for all $i\in I, r\ge 0$ along with all $h \in \h$. Note that $Y(\g)= \bigoplus_{k=0}^\infty Y(\g)[k]$ is not a grading of algebras. 

For each $n\in \mathbb{Z}_{\ge 0}$, let $Y_{\ge n}$ denote the subspace $\oplus_{k=n}^\infty Y^+[k]$ of $Y^+$, and let $(A_{n},\mathsf{q}_{n})$ consist of the (left) $Y(\g)$-module $A_{n}$ and natural quotient map $\mathsf{q}_{n}$ which are given by
\begin{equation*}
 A_{n}=Y(\g)/Y(\g)Y_{\ge n+1},\quad \mathsf{q}_{n}\colon Y(\g)\to A_{n}.  
\end{equation*}
For each $n\geq 0$, $\mathsf{q}_{n-1}$ factors through $A_{n}$ to yield a $Y(\g)$-module homomorphism $\mathsf{p}_{n}:A_{n}\to A_{n-1}$ such that $\mathsf{p}_{n}\circ \mathsf{q}_n=\mathsf{q}_{n-1}$. Therefore, $(A_n,\mathsf{p}_n)_{n\geq 0}$ forms an inverse system of $Y(\g)$-modules. Following \cite[10.1.D]{CP-book}, we introduce $\widehat Y(\g)$ as the inverse limit of this system:
\begin{Definition}\label{def:hatY}
 We define $\widehat Y(\g)$ to be the $Y(\g)$-module obtained by taking the inverse limit of the system $(A_n,\mathsf{p}_n)_{n\geq 0}$:
 \begin{equation}
  \widehat Y(\g)=\varprojlim_n A_n=\varprojlim_n (Y(\g)/Y(\g)Y_{\ge (n+1)}). \label{hatY}
 \end{equation}
\end{Definition}
Let $\mathfrak{i}\colon Y(\g)\to \widehat Y(\g)$ be the homomorphism (of $Y(\g)$-modules) given by $X\mapsto (\mathsf{q}_n(X))_{n\geq 0}$ for all $X\in Y(\g)$. Note that $\mathfrak{i}$ is injective: if $X\in \Ker(\mathfrak{i})$
then $X\in \cap_{n\geq 0}Y(\g)Y^+_{n+1} = \{ 0 \}$. 

The next lemma gives a more familiar presentation of $\widehat Y(\g)$. 
\begin{Lemma}\label{L:IL}
 The embedding $\mathfrak{i}$ extends to a linear isomorphism
 \begin{equation}
 \Phi\colon \prod_{k=0}^\infty Y(\g)[k]\to \widehat Y(\g),\quad \sum_{k=0}^\infty X_k\mapsto \left(\sum_{k=0}^n\mathsf{q}_n(X_k) \right)_{n\geq 0}.
 \end{equation} 
\end{Lemma}

Henceforth, we will always identify $\widehat Y(\g)$ and $\prod_{k=0}^\infty Y(\g)[k]$, and  we shall especially view the elements of $\widehat Y(\g)$ as infinite series
$\sum_{k=0}^\infty X_k$ with $X_k\in Y(\g)[k]$ for all $k\geq 0$. 

The main goal for the rest of this section is to prove that $\widehat Y(\g)$ can be naturally made into a $\C$-algebra with structure compatible 
with that of $Y(\g)$. We begin by naively defining what the multiplication should be.  

Given $X^\circ=\sum_{k=0}^\infty X_k^\circ$ and 
$X^\bullet=\sum_{\ell=0}^\infty X_{\ell}^\bullet$ in $\widehat{Y}(\g)$, define 
\begin{equation}
 X^\circ \cdot X^\bullet=\sum_{m=0}^\infty (X^\circ X^\bullet)_m, \label{prod}
\end{equation}
where $(X^\circ X^\bullet)_m = \sum_{k,\ell=0}^{\infty} (X^\circ_k X^\bullet_{\ell})_{m}$  and $(X^\circ_k X^\bullet_{\ell})_m$ is the component of the product $X^\circ_k X^\bullet_{\ell}$ which belongs to $Y(\g)[m]$ (note that the product $X^\circ_k X^\bullet_{\ell}$ is inside $Y(\g)$). 
To see that the right-hand side of \eqref{prod} is a well-defined element of $\widehat{Y}(\g)$, we have to show that $\sum_{k,\ell=0}^{\infty} (X^\circ_k X^\bullet_{\ell})_m$ reduces to a finite sum. This will be established in the proof of \propref{prop:prodwhY}, however first we will need \propref{prop:NmY} below whose proof depends on the next lemma. 
\begin{Lemma}\label{lem:XYz}
 For each $k\geq 0$, $i\in I$ and $r\geq 0$ we have the inclusions
 \begin{gather}
  Y^+[k]h_{ir}\subset Y(\g)[k], \label{L:1.1} \\
  Y^+[k]x_{ir}^-\subset Y(\g)[k]\oplus Y(\g)[k-1]. \label{L:2.1}
 \end{gather}
\end{Lemma}
\begin{proof}
Let's prove that the first inclusion holds for all $r\geq 0$ by induction on $k$. When $k=0$, \eqref{L:1.1} is true since $Y^+[0]=\mathbb{C}$. Assume now that, for a fixed $\ell>0$, \eqref{L:1.1} holds when $k<\ell$. Let $X=x_{i_1,r_1}^+\cdots x_{i_\ell,r_\ell}^+$ be a monomial in $Y^+[\ell]$ and write  $X=X_1 x_{i_\ell,r_\ell}^+$. We prove that $Xh_{ir} \subset Y(\g)[\ell]$ by induction on $r$. The case $r=0$ follows from \eqref{eq:relHX}. When $r>0$, we have 
 \begin{align*}
 Xh_{ir}&=X_1h_{ir}x_{i_\ell,r_\ell}^++X_1[x_{i_\ell,r_\ell}^+,h_{ir}]\\
           &=X_1h_{ir}x_{i_\ell,r_\ell}^+-X_1\left(\frac{(\alpha_i,\alpha_{i_\ell})}{2}(h_{i,r-1}x_{i_\ell,r_\ell}^++x_{i_\ell,r_\ell}^+h_{i,r-1})+(h_{i,r-1}x_{i_\ell,r_\ell+1}^+-x_{i_\ell,r_\ell+1}^+h_{i,r-1}) \right).
 \end{align*}
Since $X_1$ has length $\ell-1$, $X_1h_{ir}x_{i_\ell,r_\ell}^+$, $X_1h_{i,r-1}x_{i_\ell,r_\ell}^+$ and $X_1h_{i,r-1}x_{i_\ell,r_\ell+1}^+$ all belong to $Y(\g)[\ell]$, and, by induction on $r$, the rest of the terms on the right-hand side of the above expression also belong to $Y(\g)[\ell]$. Hence, \eqref{L:1.1} holds for all $k,r\geq 0$.

The inclusion \eqref{L:2.1} can be proved similarly using induction on $k$ and \eqref{L:1.1}.
\end{proof}

Note that, since $Y^+[k]h\subset Y(\g)[k]$ for all $h\in \h$ and $k\geq 0$, the relation \eqref{L:1.1} of \lemref{lem:XYz} implies that $Y^+[k]\cdot Y^0\subset Y(\g)[k]$ for all $k\geq 0$. We shall use this fact in the next Proposition.
\begin{Proposition}\label{prop:NmY}
Let $Z\in Y^{\le 0}$. Then, for every non-negative integer $m\geq 0$  
 there exists $N_{m}^{Z}\geq 0$ such that 
 \begin{equation}
  [Y^+[k],Z]\in \bigoplus_{a=m+1}^k  Y(\g)[a] \quad \text{ for all }\; k\geq N_m^{Z}.
 \end{equation}
\end{Proposition}
\begin{proof} Without loss of generality, we may assume that $Z$ is a monomial in the generators of $Y^{\le 0}$, say  \begin{equation}
  Z=x_{j_1,s_1}^-\cdots x_{j_\ell,s_\ell}^-H,  \label{Y:mon}
 \end{equation} 
where $H$ is a monomial in the generators of $Y^0$.
 Since $Y^+[k]$ is spanned by homogeneous monomials of degree $k$, it suffices to prove the existence of $N_{m}^{Z}\geq 0$ such that 
\begin{equation*}
   [X_k,Z]\in \bigoplus_{a=m+1}^k Y(\g)[a] \quad \text{ for all }\; k\geq N_m^{Z}
  \end{equation*}
for every monomial $X_k$ of degree $k$. We will prove the stronger result that, for $Z$ as in \eqref{Y:mon}, $N_m^{Z}$ can be taken to be precisely $m+\ell+1$.  Set $Z_b=x_{j_1,s_1}^-\cdots x_{j_b,s_b}^-$ for each $0\leq b\leq \ell$, where $Z_0$ is understood to equal $1$.
By relation \eqref{L:1.1} of \lemref{lem:XYz} we have $Y(\g)[k]\cdot H\subset Y(\g)[k]$ for all $k\geq 0$, and in particular $[X_k,H]\in \bigoplus_{a=m+1}^k Y(\g)[a]$ whenever $k\geq m+1$. These observations together with the fact that 
\begin{equation*} 
 [X_k,Z]=[X_k,Z_\ell]H+Z_\ell[X_k,H]
\end{equation*}
imply that it suffices to show that $[X_k,Z_\ell]\in \bigoplus_{a=m+1}^k Y(\g)[a]$ for all $k\geq m+\ell+1$. We will prove this statement by induction on $\ell\geq 0$. The base of the induction is immediate since $Z_{0}=1$. Next, fix $d>0$ and assume inductively that the statement holds when $\ell$ is replaced by $d-1$. If now 
$\ell$ is replaced instead by $d$, then we may rewrite $[X_k,Z_d]$ as 
\begin{equation}
 [X_k,Z_d]=Z_{d-1}[X_k,x_{j_{d},s_{d}}^-]+[X_k,Z_{d-1}]x_{j_{d},s_{d}}^- \label{eq1}. 
\end{equation}
If $d=1$, the second term on the right-hand side of the above equation vanishes and \eqref{L:2.1} of \lemref{lem:XYz} yields that $Z_{d-1}[X_k,x_{j_{d},s_{d}}^-]$ belongs to $\bigoplus_{a=m+1}^k Y(\g)[a]$ for all $k\geq m+2$, as desired. If instead $d>1$, then the latter statement of the previous sentence still holds. Moreover, the inductive hypothesis implies that 
$[X_k,Z_{d-1}]\in \bigoplus_{a=m+2}^k Y(\g)[a]$ for $k\geq m+d+1$, and thus $[X_k,Z_{d-1}]x_{j_{d},s_{d}}^-\in \bigoplus_{a=m+1}^k Y(\g)[a]$
for all such values of $k$ as a consequence of \eqref{L:2.1}. Since $m+d+1\geq m+2$, we may conclude from \eqref{eq1} that $[X_k,Z_d]\in \bigoplus_{a=m+1}^k Y(\g)[a]$ whenever $k\geq m+d+1$. 
\end{proof}
\begin{Proposition}\label{prop:prodwhY}
The operation given in \eqref{prod} is a well-defined product which equips $\widehat{Y}(\g)$ with the structure of an associative algebra.
\end{Proposition}
\begin{proof}
We have to see that $\sum_{k,\ell=0}^{\infty} (X^\circ_k X^\bullet_{\ell})_m$ reduces to a finite sum for every fixed $m$. The product $X^\circ_k X^\bullet_{\ell}$ is in $\bigoplus_{r=\ell}^{\infty} Y(\g)[r]$, so if $(X^\circ_k X^\bullet_{\ell})_m \neq 0$, then $\ell \le m$.

For each pair $k,\ell\in \Z_{\geq 0}$ write $X_k^\circ=\sum_{i\in I_k}y_{k,i} x_{k,i}^\circ$ and  $X_{\ell}^\bullet=\sum_{j\in J_{\ell}}z_{\ell,j} x^\bullet_{{\ell},j}$,
where  $y_{k,i},z_{\ell,j}\in Y^{\leq 0}$, $x_{k,i}^\circ\in Y^+[k]$ and $x_{\ell,j}^\bullet\in Y^+[\ell]$ for all $i\in I_k$, $j\in J_\ell$; $I_k$ and $J_\ell$ being finite sets. 
 We then have 
 \begin{align*}
\sum_{k,\ell=0}^{\infty} (X^\circ_k X^\bullet_{\ell})_m &=\sum_{k,{\ell}=0}^\infty\sum_{i\in I_k,j\in J_{\ell}} (y_{k,i}z_{\ell,j}x_{k,i}^\circ x_{{\ell},j}^\bullet)_m + \sum_{k,{\ell}=0}^\infty\sum_{i\in I_k,j\in J_{\ell}} (y_{k,i}[x_{k,i}^\circ,z_{\ell,j}]x_{{\ell},j}^\bullet)_m \\
                         &= \sum_{k+{\ell}=m}\sum_{i\in I_k,j\in J_{\ell}} y_{k,i}z_{\ell,j}x_{k,i}^\circ x_{{\ell},j}^\bullet 
                         +\sum_{k,{\ell}=0}^\infty\sum_{i\in I_k,j\in J_{\ell}} (y_{k,i}[x_{k,i}^\circ,z_{\ell,j}]x_{{\ell},j}^\bullet)_m.
 \end{align*}
The first sum is finite, so we need only to show that the second summation is also finite. Set \[  N= \max_{0 \le \ell\le m} \max_{j\in J_{\ell}} N_{m-\ell}^{z_{\ell,j}}. \] If $k\ge N$, then \propref{prop:NmY} implies that $[x_{k,i}^\circ,z_{\ell,j}]$ is a sum of homogeneous elements of degree $\ge m-\ell+1$ for all $0\leq \ell\leq m$. 
Therefore, 
\[ \sum_{k,{\ell}\geq 0}\sum_{i\in I_k,j\in J_{\ell}} (y_{k,i} [x_{k,i}^\circ,z_{\ell,j}]x_{{\ell},j}^\bullet)_m = \sum_{k=0}^N \sum_{\ell=0}^m \sum_{i\in I_k,j\in J_{\ell}} (y_{k,i}[x_{k,i}^\circ,z_{\ell,j}]x_{{\ell},j}^\bullet)_m \]
and the sum on the right-hand side is a finite sum. Associativity of the product on $\widehat{Y}(\g)$ is immediate.
\end{proof}
The last result of this subsection illustrates that $\widehat Y(\g)$ is particularly well-behaved with respect to the category $\mathscr O$ of $Y(\g)$.  
\begin{Proposition}\label{P:extO}
The completion $\widehat Y(\g)$ has the following properties: 
\begin{enumerate}
\renewcommand{\labelenumi}{\textup{(\theenumi)}}%
\item \label{P:comp1} For each $i\in I$, $J(x_{i}^\pm)$ and $J(h_i)$ (see \eqref{eq:Jhx}) can be viewed as elements of $\widehat Y(\g)$;
\item \label{P:comp2} Every module $V$ of $Y(\g)$ in the category $\mathscr O$ extends to a module over $\widehat Y(\g)$. 
\end{enumerate}
\end{Proposition}
\begin{proof}
As $J(x_i^\pm)=\pm(\alpha_i,\alpha_i)^{-1}[J(h_i),x_i^\pm]$, it suffices to prove \eqref{P:comp1} for $J(h_i)$, which amounts to proving that the infinite sum $\boldsymbol{v}_i=\sum_{\alpha\in \Delta_+}(\alpha,\alpha_i)\sum_{k=1}^{\dim\g_\alpha}x_{-\alpha}^{(k)}x_\alpha^{(k)}$ is contained in  $\widehat Y(\g)$. This is a consequence of the observation that $x_{-\alpha}^{(k)}x_\alpha^{(k)}\in Y(\g)[\mathrm{ht}(\alpha)]$ and $\{\alpha\in \Delta_+\,:\,\mathrm{ht}(\alpha)=\ell\}$ is finite for each $\ell\geq 1$. As for \eqref{P:comp2}, given $X=\sum_{k=0}^\infty X_k\in \widehat Y(\g)$, the operator $X_V$ given by $X_V(\mathbf{v})=\sum_{k=0}^\infty X_k \mathbf{v}$ for all $\mathbf{v}\in V$ is a well-defined element of $\End_\C V$ because $Y^+[k]\mathbf{v}=0$ for all $k \gg 0$. 
\end{proof}
\begin{Remark}\label{Rem:com} 
The definition of $\widehat Y(\g)$ as the inverse limit \eqref{hatY} has been motivated by \S 10.1.D of \cite{CP-book} and Definition 1.5.8 in  \cite{KuBook}. In \cite{CP-book}, the analogous completion $\hat U_q(\g)$ of the quantum enveloping algebra $U_q(\g)$ (where $\g$ is finite-dimensional) was introduced in order to study the universal $R$-matrix of $U_q(\g)$. A similar completion for more general quantized Kac-Moody algebras was constructed in \cite[\S 4.1]{Jo}.
 \end{Remark}

\begin{Remark}
In \cite[\S 2]{ChIl}, the authors defined an algebra $\mathfrak{U}(R,\mathscr C)$ which can be associated to any ring $R$ and a full subcategory $\mathscr C$ of the 
 category of $R$-modules. Of specific interest in \cite{ChIl} was the case where $R=U(\g)$ and $\mathscr C$ is taken to be the category $\mathscr O$ for the Kac-Moody algebra $\g$ (to this effect, see also \cite{Ku}). However, when one takes instead $R=Y(\g)$ and $\mathscr C$ to be the category $\mathscr O$ for $Y(\g)$, one arrives
 at an algebra which is closely related to $\widehat{Y}(\g)$  as a left $Y(\g)$-module, but has a different multiplication. This construction has also served as a source of motivation for our definition of $\widehat{Y}(\g)$. 
 \end{Remark}

\subsection{The coproduct $\Delta\colon Y(\g)\to Y(\g)\widehat\otimes Y(\g)$}
Let $\g$ be as in the previous subsection with $Y(\g)$ satisfying the assumptions \eqref{as.1} and \eqref{as.2}. Consider the Yangian $Y(\g\oplus \g)$. As algebras, we have the isomorphism $Y(\g\oplus \g)\cong Y(\g)\otimes Y(\g)$ (see for instance Proposition II.4.2 of \cite{Kas-book}). In particular, $Y(\g\oplus \g)$ also satisfies the assumptions \eqref{as.1} and \eqref{as.2}, and therefore we can define $Y(\g)\widehat \otimes Y(\g)$ using 
\defref{def:hatY}:
\begin{equation}
Y(\g)\widehat\otimes Y(\g)=\widehat Y(\g\oplus \g). \label{hatY2}
\end{equation}
More generally, we define the completed $n$-th tensor power $Y(\g)^{\widehat \otimes n}=Y(\g)\widehat\otimes \cdots \widehat\otimes Y(\g)$ (where $Y(\g)$ appears $n$-times) as $\widehat Y(\g^{\oplus n})$. Using \lemref{L:IL}, we can identify $Y(\g)\widehat\otimes Y(\g)$ with the direct product \[ \prod_{k=0}^\infty (Y(\g)\otimes Y(\g))[k]=\prod_{k=0}^\infty \left(\bigoplus_{r+s=k}Y(\g)[r]\otimes Y(\g)[s] \right). \]

We now return to the setting where $\g$ is an affine Lie algebra with an indecomposable Cartan matrix $(a_{ij})_{i,j\in I}$, which is not of type $A_1^{(1)}$ or $A_2^{(2)}$.

Note that the half Casimir $\Omega_+$ from \eqref{HalfCas} can be viewed as an element of $Y(\g)\widehat\otimes Y(\g)$, and therefore we may define the assignment 
\begin{equation*}
\Delta\colon \{x_{i0}^\pm, h_{i1}, h\,:\, i\in I, h\in \h\}\to Y(\g)\widehat\otimes Y(\g)
\end{equation*}
exactly as $\Delta_{V_1,V_2}$ has been defined in \eqref{eq:assign}, except that $\End_\C(V_1\otimes V_2)$ should be replaced by $Y(\g)\widehat\otimes Y(\g)$.
\begin{Proposition}\label{Prop:Del.comp} Assume that $\g$ is of affine type, but not $A_1^{(1)}$ or $A_2^{(2)}$. 
The assignment $\Delta$ extends to an algebra homomorphism $\Delta\colon Y(\g)\to Y(\g)\widehat\otimes Y(\g)$. Additionally, 
if $V_1,V_2$ belong to the category $\mathscr O$ and $\rho_1,\rho_2$ are the corresponding homomorphisms $Y(\g)\to \End_\C(V_1)$ and $Y(\g)\to \End_\C(V_2)$, respectively, then $\rho_1\otimes \rho_2$ extends to $\rho_1\widehat\otimes \rho_2\colon Y(\g)\widehat\otimes Y(\g)\to \End_\C(V_1\otimes V_2)$, and we have 
\begin{equation*}
 \Delta_{V_1,V_2}=(\rho_1\widehat\otimes \rho_2)\circ \Delta. 
\end{equation*}
\end{Proposition}
\begin{proof}
Part I of the proof of \thmref{thm:coproduct} can be carried out without modification when $\End_\C(V_1\otimes V_2)$ is replaced by $Y(\g)\widehat\otimes Y(\g)$.

For Part II of the proof of \thmref{thm:coproduct} we just need to explain why \propref{prop:Jroot} holds. Since the adjoint action of $\g$ on $Y(\g)$ is integrable (by \eqref{eq:relDS}), the formula \eqref{tau_i} now defines an algebra automorphism of $Y(\g)$ (by \cite[Lemma 1.3.5 (b)]{KuBook}, for instance). By the proof of \lemref{L:tau}, it extends to an automorphism of the subalgebra of $\widehat Y(\g)$ generated by $Y(\g)$ and $\{v_i\}_{i\in I}$, which is sufficient for our purposes. (That $v_i\in\widehat Y(\g)$ is a consequence of \propref{P:extO}.) Since \lemref{lem:vw} holds in $\widehat Y(\g)$, both \lemref{lem:tauJ} and \propref{prop:Jroot} can be proven as before using the automorphism $\tau_i$ described in the previous sentence.

If $(\rho_1,V_1)$ and $(\rho_2,V_2)$ are two representations in the category $\mathscr O$, then $(\rho_1\otimes \rho_2,V_1\otimes V_2)$ belongs to the category of $\mathscr O$ for the Yangian $Y(\g\oplus\g)$. Therefore, by \propref{P:extO}, $\rho_1\otimes\rho_2$ extends to 
a homomorphism 
\begin{equation*}
\rho_1\widehat\otimes \rho_2\colon \widehat Y(\g\oplus\g)=Y(\g)\widehat\otimes Y(\g)\to \End_\C (V_1\otimes V_2).
\end{equation*}
The equality $\Delta_{V_1,V_2}=(\rho_1\widehat\otimes \rho_2)\circ \Delta$ is now immediate since both sides agree on generators of $Y(\g)$. 
\end{proof}

\subsection{The modified Yangian}
An alternative to working with the completed tensor product $Y(\g)\widehat\otimes Y(\g)$ is to replace $\Delta$ by a family of linear maps $\Delta_{\lambda_1,\mu_1, \lambda_2,\mu_2}$ as suggested, for instance, in Chapter 23 of \cite{Lu}. This alternative also fits with the geometric construction in \cite{MaOk}, once it is formulated as in \cite{Na}. We assume again that $\g$ is affine and not of type $A_1^{(1)}$ or $A_2^{(2)}$. In particular, $\h=\mathrm{span}\, \mathcal{B}$ where $\mathcal{B} = \{ h_{i0},d \, | \, i\in I \}$ and $d$ is the derivation. Given two elements $\lambda,\mu$ of the weight lattice of $\g$, set \[ {}_{\lambda}Y(\g)_{\mu} = Y(\g)/ \left( \sum_{h \in \mathcal{B}} (h - \lambda(h)) Y(\g) + \sum_{h \in \mathcal{B}} Y(\g) (h - \mu(h)) \right) \] and let $\pi_{\lambda,\mu}\colon Y(\g) \twoheadrightarrow {}_{\lambda}Y(\g)_{\mu}$ be the projection map. Following \cite{Lu}, the non-unital algebra $\bigoplus_{\lambda,\mu} {}_{\lambda}Y(\g)_{\mu}$ could be called the modified Yangian. We will denote it $\dot{Y}(\g)$. Its algebra structure is defined as in \cite[\S 23.1.1]{Lu}: for any $\lambda_1,\mu_1,\lambda_2,\mu_2$ in the weight lattice of $\g$ and any $x_1\in Y(\g)\{\lambda_1-\mu_1\}$, $x_2\in Y(\g)\{\lambda_2-\mu_2\}$, we set $\pi_{\lambda_1,\mu_1}(x_1)\pi_{\lambda_2,\mu_2}(x_2)=\delta_{\mu_1,\lambda_2}\pi_{\lambda_1,\mu_2}(x_1x_2)$.

We have a root grading on $Y(\g)$ given by $\mathrm{deg}(x_{ir}^{\pm}) = \pm \alpha_i$, $\mathrm{deg}(h_{ir}) = 0 $ for all $i\in I, \, r\ge 0$ and $\mathrm{deg}(d)=0$, which leads to direct sum decompositions into graded pieces\[ Y(\g) = \bigoplus_{\nu \in \Z\Delta} Y(\g)\{\nu\} \text{ and } \dot{Y}(\g) = \bigoplus_{\nu \in \Z\Delta} \bigoplus_{\lambda,\mu} \pi_{\lambda,\mu}(Y(\g)\{\nu\} ). \]  Moreover, $\pi_{\lambda,\mu}(Y(\g)\{\nu\}) \neq 0 $ only if $\lambda - \mu = \nu$.

Now let $\lambda_1,\mu_1, \lambda_2,\mu_2$ be elements of the weight lattice of $\g$. The map $\pi_{\lambda_1, \mu_1} \otimes \pi_{\lambda_2,\mu_2}\colon Y(\g) \otimes Y(\g) \rightarrow {}_{\lambda_1}Y(\g)_{\mu_1} \otimes {}_{\lambda_2}Y(\g)_{\mu_2}$ can be restricted to $(Y(\g) \otimes Y(\g))[k]$ for any $k$ and we denote its restriction by $(\pi_{\lambda_1, \mu_1} \otimes \pi_{\lambda_2,\mu_2})|_k$. Set $({}_{\lambda_1}Y(\g)_{\mu_1} \otimes {}_{\lambda_2}Y(\g)_{\mu_2})[k] = (\pi_{\lambda_1, \mu_1} \otimes \pi_{\lambda_2,\mu_2})|_k ((Y(\g) \otimes Y(\g))[k])$. It can also be extended to a map  \[ \pi_{\lambda_1, \mu_1} \widehat{\otimes} \pi_{\lambda_2,\mu_2}\colon Y(\g) \widehat{\otimes} Y(\g) \rightarrow \prod_{k=0}^{\infty} ({}_{\lambda_1}Y(\g)_{\mu_1} \otimes {}_{\lambda_2}Y(\g)_{\mu_2})[k] \] by setting \[ \pi_{\lambda_1, \mu_1}  \widehat{\otimes} \pi_{\lambda_2,\mu_2} = \prod_{k=0}^{\infty} (\pi_{\lambda_1, \mu_1} \otimes \pi_{\lambda_2,\mu_2})|_k. \] 

Following \cite{Lu}, we define the linear map \[ \Delta_{\lambda_1,\mu_1, \lambda_2,\mu_2}\colon {}_{\lambda_1 + \lambda_2}Y(\g)_{\mu_1+ \mu_2} \rightarrow   \prod_{k=0}^{\infty} ({}_{\lambda_1}Y(\g)_{\mu_1} \otimes {}_{\lambda_2}Y(\g)_{\mu_2})[k]  \] by \[  \Delta_{\lambda_1,\mu_1, \lambda_2,\mu_2}(\pi_{\lambda_1 + \lambda_2, \mu_1 + \mu_2}(x)) = (\pi_{\lambda_1, \mu_1}  \widehat{\otimes} \pi_{\lambda_2,\mu_2}) (\Delta(x)). \] It turns out that the image of $\Delta_{\lambda_1,\mu_1, \lambda_2,\mu_2}$ is actually contained in $\oplus_{k=0}^{\infty} ({}_{\lambda_1}Y(\g)_{\mu_1} \otimes {}_{\lambda_2}Y(\g)_{\mu_2})[k] \cong {}_{\lambda_1}Y(\g)_{\mu_1} \otimes {}_{\lambda_2}Y(\g)_{\mu_2}$: to see this, observe that, for any fixed $\lambda_1,\mu_1,\lambda_2,\mu_2$, there are only finitely many terms of $\Delta(h_{i1})$ which are contained in $Y(\g)\{\lambda_1 - \mu_1\} \otimes Y(\g)\{\lambda_2 - \mu_2\}$ and the same is true consequently for $\Delta(x)$ for any $x \in Y(\g)$.

\section{The parameter dependent coproduct $\Delta_u$}\label{Sec:paracoprod}

In this section, we construct a parameter dependent coproduct 
$\Delta_u:Y(\g)\to (Y(\g)\otimes Y(\g)(\!(u)\!)$
from which the homomorphisms $\Delta_{V_1,V_2}$ and $\Delta$ of \thmref{thm:coproduct} and \propref{Prop:Del.comp} (assuming \eqref{as.1} and \eqref{as.2}), respectively, can be recovered. In addition to unifying the constructions of the present paper, $\Delta_u$ has been applied in \cite{GRW2} to prove the Poincar\'{e}-Birkhoff-Witt Theorem for simply laced affine Yangians.

\subsection{Definition of $\Delta_u$}

The construction of $\Delta_u$ is based on the existence of a gradation homomorphism $\mathrm{s}_u:Y(\g)\to Y(\g)[u^{\pm 1}]$, which is defined by
\begin{equation*}
 \mathrm{s}_u: x_{ir}^\pm \mapsto u^{\pm 1} x_{ir}^\pm, \quad h_{ir}\mapsto h_{ir},\quad h\mapsto h \quad \forall \; i\in I, r\geq 0 \;\text{ and }\; h\in \h.
\end{equation*}
If $\g$ is of finite type, 
then $Y(\g)$ is a Hopf algebra with a genuine coproduct $\Delta:Y(\g)\to Y(\g)\otimes Y(\g)$, and we may set 
\begin{equation*}
 \Delta_u=(1\otimes \mathrm{s}_u)\circ \Delta: Y(\g)\to (Y(\g)\otimes Y(\g))[u^{\pm 1}]. 
\end{equation*}
Now suppose that $\g$ is an affine Kac-Moody algebra which is not of type $A_1^{(1)}$ or $A_2^{(2)}$. Then we
may still formally apply $1\otimes \mathrm{s}_u$ to the right-hand side of the assignment $\Delta$ of \defref{def:Delta} to produce an assignment $\Delta_u$. That is, we let 
\begin{equation*}
 \Delta_u:\{x_{i0}^\pm, \tilde h_{i1},h\,:\,i\in I,\, h\in \h\}\to (Y(\g)\otimes Y(\g))(\!(u)\!),
\end{equation*} 
 be the assignment given by 
\begin{gather} 
\begin{split}\label{Delta_u}
 \Delta_u(x_{i0}^\pm)=x_{i0}^\pm \otimes 1 + 1\otimes x_{i0}^\pm u^{\pm 1}, \quad \Delta(h)=\square(h),\\
 \Delta_u(\tilde h_{i1})=\square (\tilde h_{i1})-\sum_{k=1}^\infty\left( \sum_{\alpha\in \Delta_+^\mathrm{re}(k)}(\alpha,\alpha_i)x_\alpha^-\otimes x_\alpha^+ \right) u^{k},
 \end{split}
\end{gather}
for all $i\in I$ and $h\in \h$, where $\Delta_+^{\mathrm{re}}(k)=\{\alpha\in \Delta_+^{\mathrm{re}}\,:\,\mathrm{ht}(\alpha)=k\}$ as in \secref{sec:proofII}. 
\begin{Theorem}\label{thm:cop_u}
  Assume $\g$ is either finite-dimensional (but not $\algsl_2$) or of affine type (but not of type $A_1^{(1)}$ or $A^{(2)}_2$).
  Then the assignment $\Delta_u$ defines an algebra homomorphism $\Delta_u
  \colon Y(\g)\to (Y(\g)\otimes Y(\g))(\!(u)\!)$.
\end{Theorem}
The proof of \thmref{thm:cop_u}, to be given in \subsecref{cop_u:pf}, is based on the arguments employed to prove \thmref{thm:coproduct}. The added difficulty lies in making sense of the operators of \subsecref{S:YopcatO}, which play an essential argument in the Part II of the proof of  \thmref{thm:coproduct}. Given the assumptions \eqref{as.1} and \eqref{as.2}, this can be addressed using the completion $\widehat{Y}(\g)$ as in the previous subsection. To avoid these assumptions, we will construct a $\g$-module $\widetilde{Y}(\g)$ containing $Y(\g)$ as a submodule. 

\subsection{The $\g$-module $\widetilde{Y}(\g)$}

Let $\mathfrak{n}_\pm$ be the Lie subalgebras of $\g$ generated by $\{x_i^\pm\}_{i\in I}$ and set $\mathfrak{b}_-=\mathfrak{n}_-\oplus \h$. Since $U(\g)$ admits a multiplicative triangular decomposition 
\begin{equation*}
 U(\g)\cong U(\mathfrak{n}_-)\otimes U(\h)\otimes U(\mathfrak{n}_+)\cong U(\mathfrak{b}_-)\otimes U(\mathfrak{n}_+)
\end{equation*}
and the algebra $U(\mathfrak{n}_+)$ has a $\Z_{\geq 0}$-grading $U(\mathfrak{n}_+)=\bigoplus_{k\geq 0}U(\mathfrak{n}_+)[k]$ given by $\deg x_i^+=1$ for all $i\in I$, we have the vector space decomposition $U(\g)=\bigoplus_{k\geq 0}U(\g)[k]$ with $U(\g)[k]=U(\mathfrak{b}_-)\otimes U(\mathfrak{n}_+)[k]$. Set 
\begin{equation*}
\widehat U(\g)=\prod_{k\geq 0}U(\g)[k].
\end{equation*}
By \cite[\S 1.5]{KuBook} (or the arguments of \S 5(i)), $\widehat U(\g)$ inherits from $U(\g)$ the structure of an associative algebra. In addition,
the formulas \eqref{eq:v_i}, \eqref{eq:w_i} and \eqref{v_beta} give rise to well-defined families of elements $\{v_i,w_i^\pm\}_{i\in I},\{\boldsymbol{v}_\beta\}_{\beta\in \Delta}\subset \widehat U(\g)$. Note that, under our current assumptions, $\boldsymbol{v}_\beta$ coincides with  $\sum_{\alpha\in \Delta_+^{\mathrm{re}}}(\alpha,\beta)x_\alpha^-x_\alpha^+\in \widehat U(\g)$. As before, we denote $\boldsymbol{v}_{\alpha_i}$ simply by $\boldsymbol{v}_i$. 
 
 We may view both $Y(\g)$ and $\widehat U(\g)$ as $\g$-modules equipped with the adjoint action. Their direct sum $Y(\g)\oplus \widehat U(\g)$ then contains the submodule $\mathbf{U}=\{\iota(x)-x\,:\,x\in U(\g)\}$ (see \eqref{eqn:iota}).
\begin{Definition}
 The $\g$-module $\widetilde Y(\g)$ is defined as the $\g$-submodule of $(Y(\g)\oplus \widehat U(\g))/\mathbf{U}$ generated by the images of $Y(\g)$ and $\{\boldsymbol{v}_i\}_{i\in I}$. 
\end{Definition}
\begin{Remark}
 It is worth explaining the motivation behind the above definition. As will become evident in  \subsecref{cop_u:pf}, the main obstacle to adapting Part II of the proof of \thmref{thm:coproduct} is establishing the identity \eqref{Lt'} below, which concerns the adjoint action of $\g$ on $Y(\g)$. In the proof of \thmref{thm:coproduct}, \propref{prop:Jroot} was employed to reduce this relation to \eqref{Lt}, which is perfectly valid in  $\widehat U(\g)$, but does not have a meaning in $Y(\g)$. A solution to overcoming this obstacle without the assumptions \eqref{as.1} and \eqref{as.2} is thus to construct a $\g$-module which 
 \begin{enumerate}[(a)]
  \item contains $Y(\g)$ (viewed as a $\g$-module with the adjoint action) as a submodule,
  \item contains an image of the algebra $\widehat U(\g)$ where \eqref{Lt} is valid,
  \item and where \propref{prop:Jroot} (a $\g$-module statement) is still valid.
 \end{enumerate}
As \lemref{L:delta_u} below proves, $\widetilde Y(\g)$ has all of these properties. 
\end{Remark}
We let $\mathsf{ad}:\g\to \End_\C(\widetilde Y(\g))$ denote the homomorphism of Lie algebras responsible for the $\g$-module structure on $\widetilde Y(\g)$. 

For each $i\in I$, define $J(h_i)\in \widetilde Y(\g)$ to be the image of $h_{i1}+v_i$ in $\widetilde Y(\g)$ and set \[J(x_i^\pm)=\pm(\alpha_i,\alpha_i)^{-1}[x_i^\pm,J(h_i)]\in \widetilde Y(\g).\] Slightly abusing notation, we will denote the images of $v_{i}$, $w_i^\pm$ and $ \boldsymbol{v}_\beta$ in $\widetilde Y(\g)$ again by $v_i$, $w_i^\pm$ and $\boldsymbol{v}_\beta$, respectively. Similarly, for each $X\in Y(\g)$, the image of $X$ under the natural map $Y(\g)\to \widetilde Y(\g)$ (which by the next lemma is injective) will be again denoted by $X$.  
\begin{Lemma}\label{L:delta_u}
 \leavevmode
 \begin{enumerate}
  \item \label{L:du-1} The natural homomorphism of $\g$-modules $Y(\g)\to \widetilde Y(\g)$ is an embedding,
  \item \label{L:du-2} For each $i\in I$, $\tau_i=\exp(\mathsf{ad}(e_i))\exp(-\mathsf{ad}(f_i))\exp(\mathsf{ad}(e_i))\in \End_{\C}(\widetilde Y(\g))$,
  \item \label{L:du-3} The statements of \lemref{lem:vw}, \lemref{lem:tauJ}, and \propref{prop:Jroot} hold in $\widetilde Y(\g)$. 
 \end{enumerate}

\end{Lemma}
\begin{proof}
Part \eqref{L:du-1} is straightforward. Consider Part \eqref{L:du-2}. 
By the same reasoning as given in \subsecref{S:commrel} (see \eqref{tau_i}), $\tau_i$ defines a linear automorphism of the image of $Y(\g)$ in $\widetilde Y(\g)$ which extends to all of $\widetilde Y(\g)$ since it sends $\boldsymbol{v}_j$ to $\boldsymbol{v}_{s_i(\alpha_j)}+(\alpha_i,\alpha_j)\{x_{i0}^-,x_{i0}^+\}$ (see \lemref{L:tau}) and $\widetilde Y(\g)$ is generated by $Y(\g)\cup\{\boldsymbol{v}_i\}_{i\in I}$.

Additionally, we observe that, by Lemma 1.3.5 (b) of \cite{KuBook}, 
\begin{equation} \label{tau:alg}
 \tau_i([x,y])=[\tau_i(x),\tau_i(y)] \quad \forall \; x\in \g,\; y\in \widetilde Y(\g). 
\end{equation}
Now let us turn to Part \eqref{L:du-3} of the lemma. The relations of \lemref{lem:vw} hold in the algebra $\widehat U(\g)$, and hence also in the $\g$-module $\widetilde Y(\g)$. Consequently, the proof of \lemref{lem:tauJ} remains valid after replacing $\End_\C(V)$ with $\widetilde Y(\g)$ (and with $\tau_i$ as in Part \eqref{L:du-2}). Moreover, although $\tau_i$ is not an algebra homomorphism (as $\widetilde Y(\g)$ is not an algebra), the property \eqref{tau:alg} guarantees that the proof of \propref{prop:Jroot} is still valid in our current setting. \qedhere
\end{proof}

\subsection{Proof of \thmref{thm:cop_u}}\label{cop_u:pf} 
Consider the formal series 
\begin{equation*}
 \Omega_+(u)\defeq\sum_{m=1}^{\dim \h}h_{(m)}\otimes h^{(m)}+\sum_{k=1}^\infty\left( \sum_{\alpha\in \Delta_+^\mathrm{re}(k)}x_\alpha^-\otimes x_\alpha^+ \right) u^{k}\in (Y(\g)\otimes Y(\g))[\![u]\!].
\end{equation*}
One proves exactly as in the proof of \lemref{lem:sq} that $\Omega_+(u)$ satisfies 
\begin{align*}
  & [\square (h), \Omega_+(u)] = 0 \quad \text{for $h\in\mathfrak h$},
\\
  & [\square_u(x_i^+), \Omega_+(u)] = - x_i^+ \otimes h_i, 
\\
  & [\square_u(x_i^-), \Omega_+(u)] =  h_i\otimes x^-_i u^{-1},
\end{align*}
where $\square_u(x_i^\pm)=x_i^\pm \otimes 1 +1\otimes x_i^\pm u^{\pm 1}$ for each $i\in I$. 

The proof that the assignment $\Delta_u$ preserves the relations of \thmref{thm:deduction}, excluding \eqref{eq:relHH'} with $r=s=1$, is now achieved as in Part I of the proof of \thmref{thm:coproduct} after inserting powers of $u$ in appropriate places and replacing the role of \lemref{lem:sq} by the above commutation relations for $\Omega_+(u)$. 

To complete the proof, we must see why $[\Delta_u(\tilde h_{i1}),\Delta_u(\tilde h_{j1})]=0$ for all $i,j\in I$. As $\Delta_u(\tilde h_{k1})=\square(\tilde h_{k1})+[h_{k0}\otimes 1,\Omega_+(u)]$ and $[\tilde h_{i1},\tilde h_{j1}]=0$, it suffices to prove that
\begin{equation*}
 [[h_{i0}\otimes 1,\Omega_+(u)],[h_{j0}\otimes 1,\Omega_+(u)]]
 =[\square(\tilde h_{j1}),[h_{i0}\otimes1,\Omega_+(u)]]-[\square(\tilde h_{i1}),[h_{j0}\otimes1,\Omega_+(u)]]. \label{suff'}
\end{equation*}
However, the $u^k$ coefficient of the left-hand side of this expression is precisely \eqref{LHS}, and the $u^k$ coefficient of the right-hand side is \eqref{RHS} (now both viewed inside $Y(\g)\otimes Y(\g)$). As before, proving the equality of \eqref{LHS} and \eqref{RHS} reduces to proving the identity
\begin{equation}\label{Lt'}
 [h_{ij}(\alpha),x_\alpha^\mp]=\frac{1}{2}\sum_{\substack{\beta,\gamma\in \Delta_+^{\mathrm{re}}\\ \beta+\gamma=\alpha}}A_{ij,\beta,\gamma}(x_\alpha^\mp,[x_\beta^\pm,x_\gamma^\pm])x_\beta^\mp x_\gamma^\mp \quad \forall \; \alpha\in \Delta_+^{\mathrm{re}},
 \end{equation}
 where $h_{ij}(\alpha)=(\alpha,\alpha_j)\tilde h_{i1}-(\alpha,\alpha_i)\tilde h_{j1}$: see \eqref{hij=vij} and \eqref{Lt}. 
 By Part \eqref{L:du-1} of \lemref{L:delta_u}, it suffices to prove \eqref{Lt'} in the $\g$-module $\widetilde Y(\g)$. Since, by Part \eqref{L:du-3} of \lemref{L:delta_u}, \propref{prop:Jroot} holds in $\widetilde Y(\g)$, the equality \eqref{hij=vij} is satisfied in $\widetilde Y(\g)$:
 \begin{equation*}
  [h_{ij}(\alpha),x_\alpha^\mp]=[x_\alpha^\pm,v_{ij}(\alpha)],
 \end{equation*}
where $v_{ij}(\alpha)=(\alpha,\alpha_j)\tilde v_{i}-(\alpha,\alpha_i)\tilde v_{j}$. It is thus enough to prove that \eqref{Lt}, i.e. \eqref{Lt'} with the left-hand side replaced by $[x_\alpha^\pm,v_{ij}(\alpha)]$, holds in $\widetilde Y(\g)$.
The argument which was used to establish \eqref{Lt} in Part II of the proof of \thmref{thm:coproduct} also proves that \eqref{Lt} holds in $\widehat U(\g)$, and it therefore also holds in $\widetilde Y(\g)$.

\subsection{Recovering $\Delta_{V_1,V_2}$ and $\Delta$} Let $V_1$ and $V_2$ be two arbitrary $Y(\g)$-modules with associated algebra homomorphisms
$\rho_a:Y(\g)\to \End_\C(V_a)$ for $a\in \{1,2\}$. The tensor product $\rho_1\otimes \rho_2$ extends to a homomorphism
\begin{gather*}
\rho_1\otimes_u\rho_2:(Y(\g)\otimes Y(\g))(\!(u)\!)\to (\End_\C(V_1\otimes V_2))(\!(u)\!)\subset \End_{\C(\!(u)\!)}\left((V_1\otimes V_2)(\!(u)\!)\right),\\
\sum_{k\in \Z}Y_k u^k\mapsto \sum_{k\in \Z}(\rho_1\otimes \rho_2)(Y_k)u^k,
\end{gather*}
and hence, by \thmref{thm:cop_u}, the composition $(\rho_1\otimes_u\rho_2)\circ \Delta_u$ equips $(V_1\otimes V_2)(\!(u)\!)$ with the structure of a $Y(\g)$-module. 

If $(V_1\otimes V_2)[u^{\pm 1}]$ is a submodule of $(V_1\otimes V_2)(\!(u)\!)$, then the image of $(\rho_1\otimes_u\rho_2)\circ \Delta_u$ in $\End_{\C(\!(u)\!)}\left((V_1\otimes V_2)(\!(u)\!)\right)$ may be viewed as a subalgebra of  $\End_{\C[u^{\pm 1}]}\left((V_1\otimes V_2)[u^{\pm 1}]\right)$. The $\C[u^{\pm 1}]$-module evaluation map  
$V[u^{\pm 1}]\to V$, $u\mapsto 1$, then induces an algebra homomorphism
\begin{equation}
 \mathrm{ev}:\End_{\C[u^{\pm 1}]}\left((V_1\otimes V_2)[u^{\pm 1}]\right)\to \End_\C(V_1\otimes V_2),
\end{equation}
and the composition $\mathrm{ev}\circ (\rho_1\otimes_u\rho_2)\circ \Delta_u:Y(\g)\to \End_\C(V_1\otimes V_2)$ makes $V_1\otimes V_2$ a genuine $Y(\g)$-module. 

The next corollary illustrates that the category $\mathscr{O}$ provides a source of modules which fit into this framework. In addition, it demonstrates that $\Delta_u$ induces the algebra homomorphisms $\Delta_{V_1,V_2}$ and $\Delta$ of \thmref{thm:coproduct} and \propref{Prop:Del.comp}, respectively.
\begin{Corollary}\label{C:Delta_u}  Let $V_1$ be an arbitrary $Y(\g)$-module, fix $V_2\in \mathscr{O}$, and recall that $\Delta:Y(\g)\to Y(\g)\widehat\otimes Y(\g)$ is the algebra homomorphism of \propref{Prop:Del.comp}. Then
 \begin{enumerate}
  \item \label{Delta_u:1}$(V_1\otimes V_2)[u^{\pm 1}]$ is a submodule of $(V_1\otimes V_2)(\!(u)\!)$ and hence $\mathrm{ev}\circ (\rho_1\otimes_u\rho_2)\circ \Delta_u$ makes $V_1\otimes V_2$ into a $Y(\g)$-module. 
  \item\label{Delta_u:2} If in addition $V_1\in \mathscr{O}$, then the homomorphism $\Delta_{V_1,V_2}$ of \thmref{thm:coproduct} is recovered by
  \begin{equation*}
   \Delta_{V_1,V_2}=\mathrm{ev}\circ (\rho_1\otimes_u\rho_2)\circ \Delta_u.
  \end{equation*}
  \item \label{Delta_u:3} Assume that \eqref{as.1} and \eqref{as.2} are satisfied and let $\iota_u$ denote the natural inclusion $(Y(\g)\otimes Y(\g))(\!(u)\!)\to (Y(\g)\widehat \otimes Y(\g))(\!(u)\!)$. Then, evaluation at $u=1$ produces an algebra homomorphism 
  \begin{equation*}
  \widehat{\mathrm{ev}}:\mathrm{Image}(\iota_u\circ \Delta_u)\to Y(\g)\widehat \otimes Y(\g),
  \end{equation*}
  and we have $\Delta=\widehat{\mathrm{ev}}\circ \iota_u\circ \Delta_u$.  
 \end{enumerate} 
\end{Corollary}
\begin{proof}
 If $V_2\in \mathscr{O}$, then for any fixed $\mathbf{v}\in V_2$ there is $m\geq 0$ such that $x_\alpha^+(\mathbf{v})=0$ for all $\alpha\in \Delta_+^\mathrm{re}(k)$ with $k>m$.
 Therefore, for any fixed $Y(\g)$-module $V_1$ and element $\mathbf{w}\in V_1\otimes V_2$, $\Delta_u(\tilde h_{i1})(\mathbf{w})\in (V_1\otimes V_2)[u]$ for each $i\in I$ (see \eqref{Delta_u}). This implies that $(V_1\otimes V_2)[u^{\pm 1}]$ is indeed a submodule of $(V_1\otimes V_2)(\!(u)\!)$. Part \eqref{Delta_u:1} thus follows from the discussion preceding the statement of the corollary and Part \eqref{Delta_u:2} is deduced by comparing the definitions of $\Delta_u$ and $\Delta_{V_1,V_2}$ (see \eqref{eq:assign} and \eqref{Delta_u}).
 
 Part \eqref{Delta_u:3} of the corollary is a consequence of the fact that the application of $\widehat{\mathrm{ev}}$ to the generators 
 $\iota_u(\Delta_u(\tilde h_{i1}))$, $\iota_u(\Delta_u(x_{i0}^\pm))$ and $\iota_u(\Delta_u(h))$ (for $i\in I$ and $h\in \h$) of $\mathrm{Image}(\iota_u\circ \Delta_u)$ produces the well-defined elements $\Delta(\tilde h_{i1}), \Delta(x_{i0}^\pm)$ and $\Delta(h)$, respectively, of $Y(\g)\widehat\otimes Y(\g)$. \qedhere
\end{proof}
\begin{Remark} 
By \propref{Prop:mon}, $\Delta_{V_1,V_2}$ is coassociative. On the other hand,  $\Delta_u$  satisfies the same ``twisted'' coassociativity property as the deformed Drinfeld (or ``non-standard") coproduct for quantum affinizations \cite[Lemma 3.4]{He2}: 
 \begin{equation*}
(\Delta_u\otimes \mathrm{id}) \circ \Delta_{uv}=(\mathrm{id}\otimes \Delta_v) \circ \Delta_{u}.
 \end{equation*}
The use of the gradation homomorphism $\mathrm{s}_u$ to ``discover'' $\Delta_u$ has been inspired by similar ideas employed by Hernandez to develop
 the deformed Drinfeld coproduct in \cite{He1}: see Remark 2 therein. Note, however, that the Yangian analogue of the gradation morphism $T$ in \cite[Remark 2]{He1} specializes (for $u\in \C$) to a shift automorphism of the Yangian \cite[\S 4.5]{GTL3}, which differs significantly from $\mathrm{s}_u$.
\end{Remark}

\section{Two parameter Yangian in type $A^{(1)}_{n-1}$}\label{sec:2para}

In this section, we assume that $\g$ is of type $A^{(1)}_{n-1}$ and $n\ge 3$. (\defref{def:TwoPar} below is not the correct one when $n=2$: in this case, see the definition in \S 1.2 in \cite{BeTs} and Definition 5.1 in \cite{Ko}.) We identify the index set $I$ with $\Z/n\Z$ and normalize $(\ ,\ )$ so that $(\alpha_i,\alpha_i) = 2$ for all $i\in I$. In this case, the definition of the Yangian $Y_\hbar(\g)$ can be generalized by introducing a second parameter $\ve$ (see \cite{Gu2}; for quantum toroidal algebras, see \cite{VaVa}).
\begin{Definition}\label{def:TwoPar}
Let $\hbar,\ve\in\C$. The Yangian $Y_{\hbar,\ve}(\g^\prime)$ is the algebra over $\C$ with generators $x_{ir}^{\pm}, \, h_{ir}$ ($i\in I,\, r\in\Z_{\ge 0}$) subject to 
the defining relations of $Y_\hbar(\g^\prime)$ given in \defref{def:Yangian} with the modification that, when $j=i+1$ or $j=i-1$, \eqref{eq:relexHX} and \eqref{eq:relexXX} are replaced with the 
relations: 
\begin{gather}
[h_{i,r+1}, x^\pm_{i+1,s}] - [h_{i,r}, x^\pm_{i+1,s+1}] = 
  \mp \frac\hbar2 \left(x^\pm_{i+1,s} h_{ir} + h_{ir} x^\pm_{i+1,s}\right)
  + \frac{\ve}2 [h_{i r}, x^\pm_{i+1, s}], \label{eq:relHXve1}
  \\
  [h_{i,r+1}, x^\pm_{i-1,s}] - [h_{i,r}, x^\pm_{i-1,s+1}] = 
  \mp \frac\hbar2 \left( x^\pm_{i-1,s} h_{ir} + h_{ir} x^\pm_{i-1,s}\right)
  - \frac{\ve}2 [h_{i r}, x^\pm_{i-1,s}], \label{eq:relHXve2}
  \\
  [x^\pm_{i, r+1},x^\pm_{i+1,s}] - [x^\pm_{i r}, x^\pm_{i+1,s+1}] = 
  \mp \frac\hbar2 \left(x^\pm_{i+1,s} x^\pm_{i r} 
    + x^\pm_{i r} x^\pm_{i+1,s}\right)
  + \frac{\ve}2 [x^\pm_{i r}, x^\pm_{i+1,s}]. \label{eq:relXXve}
\end{gather}

The Yangian $Y_{\hbar,\ve}(\g)$ is then defined in the same manner as $Y_{\hbar}(\g)$: it is the quotient of $Y_{\hbar,\ve}(\g^\prime)\otimes_\C U(\h)$ by the ideal generated by the relations \eqref{eq:Cartan}. 
\end{Definition}

The defining relations for $Y_{\hbar,\ve}(\g^\prime)$ given above are slightly different from those which appear in \cite[Def.~2.3]{Gu2} (where $Y_{\hbar,\ve}(\g^\prime)$ is denoted by $\hat{\mathbf{Y}}_{\beta,\lambda}$).
One advantage of the relations in \defref{def:TwoPar} is that they are invariant under the rotational symmetry $0\to 1\to 2\to \cdots\to n-1\to 0$. 
Both definitions are equivalent as follows. We set
\begin{equation*}
  \begin{split}
   h'_{i r} & \defeq \sum_{s=0}^r \binom{r}{s}
   \ve^{r-s} \left(\frac{i}2
     - \frac{n}{4}\right)^{r-s} h_{i s},
\\
   x^{\prime\pm}_{i r} & \defeq \sum_{s=0}^r \binom{r}{s} 
   \ve^{r-s} \left(\frac{i}2
     - \frac{n}{4}\right)^{r-s} x^\pm_{i s}
  \end{split}
\end{equation*}
for $i=1,2,\dots, n-1$ and $x^{\prime \pm}_{0r} = x^{\pm}_{0r}$, $h_{0r}^\prime = h_{0r}$.
Then $x^{\prime\pm}_{i r}$ and $h^{\prime}_{i r}$  satisfy the relations in \cite[Def.~2.3]{Gu2} with $\lambda = \hbar$, $\beta = -\frac{\ve n}{4} + \frac{\hbar}{2}$.

When $\ve\neq 0$,  $Y_{\hbar=0,\ve}(\g^\prime)$ is isomorphic to the enveloping algebra of the universal central extension of the Lie algebra of $n\times n$ matrices with entries in the ring of differential operators on $\C^{\times}$: see \S 5 in \cite{Gu2}. Otherwise, if $\hbar_1\neq 0$ and $\hbar_2\neq 0$, then $Y_{\hbar_1,\ve}(\g) \cong Y_{\hbar=1,\ve/\hbar_1}(\g) \cong Y_{\hbar_2,\ve\hbar_2 / \hbar_1}(\g)$, so it is enough to focus on $Y_{\hbar=1,\ve}(\g)$ for any $\ve\in\C$. 

Our goal for the rest of this paper is to explain how the main results established in the previous sections also hold for $Y_{\hbar=1,\ve}(\g)$ after making only a few minor adjustments.

We begin by noting that it has already been proven in \cite{Gu2} that \thmref{thm:deduction} holds for $Y_{\hbar=1,\ve}(\g)$ with \eqref{eq:relexHX2'} and \eqref{eq:relexXX'} replaced by \eqref{eq:relHXve1}, \eqref{eq:relHXve2} and \eqref{eq:relXXve} 
with $r=s=0$ when $j=i +1$ or $j=i-1$: see Proposition 2.1 in \textit{loc.\ cit.} 

It is also the case that \thmref{thm:coproduct} holds for $Y_{\hbar=1,\ve}(\g)$ with $\Delta$ given by the same formula \eqref{eq:assign}. 
The proof of \thmref{thm:coproduct} in this case follows the same steps as before, except that some new terms appear due to the presence of the second parameter $\ve$.  The remainder of this section will be devoted to explaining the key differences and necessary modifications. We start by introducing operators $J(h_i)$ and $J(x_i^{\pm})$ on modules in the category $\mathscr{O}$ exactly as in \eqref{eq:Jhx}. As a consequence of \lemref{lem:vw}, these operators still satisfy the equivalences \eqref{eq:equiv}, 
however, the second and fourth equivalences should be altered when $j=i+ 1$ or $j=i-1$ to account for the modified relations of \defref{def:TwoPar}.
It is straightforward to verify that \eqref{eq:relHXve1}, \eqref{eq:relHXve2} and \eqref{eq:relXXve} with $(r,s)=(0,0)$ are equivalent to  the relations 
\begin{gather} [J(h_i),x_{i+1}^\pm]=\pm(\alpha_i,\alpha_{i+1})(J(x_{i+1}^\pm)+\frac{\ve}{2}x_{i+1}^\pm), \label{equiv.1} \\
[J(h_i),x_{i-1}^\pm]=\pm(\alpha_i,\alpha_{i-1})(J(x_{i-1}^\pm)-\frac{\ve}{2}x_{i-1}^\pm), \label{equiv.2}\\
[J(x_i^\pm),x_{i+1}^\pm]=[x_{i}^\pm,J(x_{i+1}^\pm)+\frac{\ve}{2}x_{i+1}^\pm] , \label{equiv.3}
\end{gather}
respectively. To account for these changes, \lemref{lem:tauJ} has to be slightly modified as follows. 
\begin{Lemma}\label{Lem:t(J)}
We have 
 \begin{equation*}
  \tau_i(J(h_j))=J(s_i(h_j))-\frac{\ve}{2}(\delta_{i+1,j} - \delta_{i-1,j})h_{i}=J(h_j)-(\alpha_i,\alpha_j)J(h_i)+\frac{\ve}{2}(\delta_{i,j+1} - \delta_{i,j-1})h_{i}
 \end{equation*}
for all $i,j\in I$. 
\end{Lemma}
The proof of this lemma is the same as for \lemref{lem:tauJ}. The operators $J(x_{\alpha}^{\pm})$ are also defined as before (see \eqref{eq:Jxal}), but \propref{prop:Jroot} has to be modified to account for the second parameter $\ve$.
\begin{Proposition}\label{Prop:NewJhx}
 For every positive real root $\alpha$ and every $i\in I$, there exists an integer $c_{\alpha,i}$ such that 
 \begin{equation}
  [J(h_i),x_\alpha^\pm]=\pm(\alpha_i,\alpha)J(x_\alpha^\pm)\pm \frac{\ve}{2}c_{\alpha,i}x_\alpha^\pm=[h_{i},J(x_\alpha^\pm)]\pm \frac{\ve}{2}c_{\alpha,i}x_\alpha^\pm.
 \end{equation}
\end{Proposition}
\begin{proof}
 We employ the same strategy as was used in the proof of \propref{prop:Jroot}: we argue by induction on $p$, where $x_\alpha^\pm=\tau_{i_1}\tau_{i_2}\cdots \tau_{i_{p-1}}(x_{i_p}^\pm)$ with $x_{i_p}^\pm=x_j^\pm$ for some $j\in I$. If $p=1$, 
 then $x_\alpha^\pm=x_j^\pm$ and the equivalences \eqref{equiv.1}, \eqref{equiv.2} and \eqref{eq:equiv} imply that we have 
 \begin{equation}
  [J(h_i),x_{j}^\pm]=\pm (\alpha_i,\alpha_j)J(x_j^\pm)\pm (\delta_{i,j+1} - \delta_{i,j-1})\frac{\ve}{2}x_{j}^\pm=[h_{i},J(x_j^\pm)]\pm (\delta_{i,j+1} - \delta_{i,j-1})\frac{\ve}{2}x_{j}^\pm, \label{NewJx}
 \end{equation}
and hence we may take $c_{\alpha_j,i}=\delta_{i,j+1} - \delta_{i,j-1}$. Suppose that the proposition holds for $x_\beta^\pm$ with $\beta=s_{i_2}\cdots s_{i_{p-1}}(\alpha_{i_p})$.  \lemref{Lem:t(J)}, with $i_1,i$ instead of $i,j$,  gives \begin{align*} J(h_i) = {} & \tau_{i_1}^{-1}(J(h_i)) - (\alpha_{i_1},\alpha_i) \tau_{i_1}^{-1}(J(h_{i_1}) ) + \frac{\ve}{2} c_{\alpha_i,i_1} \tau_{i_1}^{-1}(h_{i_1}) \\
 = {} & \tau_{i_1}^{-1}(J(h_i)) + (\alpha_{i_1},\alpha_i) J(h_{i_1}) - \frac{\ve}{2} c_{\alpha_i,i_1} h_{i_1}.
\end{align*} Therefore,
\begin{align*}
 [J(h_i),x_\alpha^\pm]&=\tau_{i_1}([\tau_{i_1}^{-1}J(h_i),x_\beta^\pm])\\
					 &=\tau_{i_1}\left([J(s_{i_1}(h_i))+\frac{\ve}{2}c_{\alpha_i,i_1}h_{i_1},x_\beta^\pm] \right)\\
                      &=\tau_{i_1}\left([J(h_i)-(\alpha_{i_1},\alpha_i)J(h_{i_1}),x_\beta^\pm] \right)\pm (\alpha_{i_1},\beta)c_{\alpha_i,i_1}\frac{\ve}{2}\tau_{i_1}(x_\beta^\pm)\\
                      &=\pm(s_{i_1}(\alpha_i),\beta)\tau_{i_1}(J(x_\beta^\pm))\pm\frac{\ve}{2}(c_{\beta,i}-(\alpha_{i_1},\alpha_i)c_{\beta,i_1}+(\alpha_{i_1},\beta)c_{\alpha_i,i_1}\tau_{i_1}(x_\beta^\pm)\\
                      &=\pm(\alpha_i,\alpha)J(x_\alpha^\pm)\pm\frac{\ve}{2}\left( c_{\beta,i}-(\alpha_{i_1},\alpha_i)c_{\beta,i_1}+(\alpha_{i_1},\beta)c_{\alpha_i,i_1}\right)x_{\alpha}^\pm,
\end{align*}
where the third equality uses the induction assumption and the fourth equality uses that $(s_{i_1}(\alpha_i),\beta)=(\alpha_i,s_{i_1}(\beta))=(\alpha_i,\alpha)$. Setting 
$c_{\alpha,i}=c_{\beta,i}-(\alpha_{i_1},\alpha_i)c_{\beta,i_1}+(\alpha_{i_1},\beta)c_{\alpha_i,i_1}$ we obtain that $[J(h_i),x_\alpha^\pm]=\pm (\alpha_i,\alpha)J(x_\alpha^\pm)\pm \frac{\ve}{2}c_{\alpha,i}x_\alpha^\pm$. 
Moreover, by the induction assumption we have the equalities $[J(h_i),x_\beta^\pm]=[h_{i},J(x_\beta^\pm)]\pm \frac{\ve}{2}c_{\beta,i}x_\beta^\pm$ and $[J(h_{i_1}),x_\beta^\pm]=[h_{i_1},J(x_\beta^\pm)]\pm \frac{\ve}{2}c_{\beta,i_1}x_\beta^\pm$.
These together with the second equality in the expansion of $[J(h_i),x_\alpha^\pm]$ above also imply that 
\begin{align*}
 [J(h_i),x_\alpha^\pm] = {} & \tau_{i_1}\left([\tau_{i_1}^{-1}(h_i),J(x_\beta^\pm)] \right)  \pm\frac{\ve}{2}\left( c_{\beta,i}-(\alpha_{i_1},\alpha_i)c_{\beta,i_1}+(\alpha_{i_1},\beta)c_{\alpha_i,i_1}\right)x_{\alpha}^\pm \\
 = {} & [h_{i},J(x_\alpha^\pm)]\pm \frac{\ve}{2}c_{\alpha,i}x_\alpha^\pm. \qedhere
\end{align*}
\end{proof}
Part I of the proof of \thmref{thm:coproduct} is the same as before except for new terms involving $\ve$ which appear when computing 
$[\Delta(\tilde{h}_{i1}),\Delta (x_{j}^{\pm})]$ and $[\Delta (x_{i1}^{\pm}),\Delta (x_{j}^{\pm})]$ when $j=i\pm 1$.

Only the following two modifications must be made in Part II of the proof of \thmref{thm:coproduct}. First, the relation \eqref{suff:2} should be replaced with 
 \begin{equation}
 \label{suff:2'}
  \sum_{\mathrm{ht}(\alpha+\beta)=k}(\alpha_i,\alpha)(\alpha_j,\beta)\{x_\alpha^\mp,x_\beta^\mp\}\otimes [x_\alpha^\pm,x_\beta^\pm]
  =2\hspace{-.5em}\sum_{\alpha\in \Delta_+^{\mathrm{re}}(k)}\left([h_{ij}(\alpha),x_\alpha^\mp]\pm d_{\alpha,i,j}\frac{\ve}{2}x_\alpha^\mp\right)\otimes x_\alpha^\pm,
 \end{equation}
where $d_{\alpha,i,j}=(\alpha,\alpha_j)c_{\alpha,i}-(\alpha,\alpha_i)c_{\alpha,j}$. Note that, by \eqref{LHS} and \eqref{RHS}, this still implies the relation \eqref{suff}. The above adjustment is made in anticipation of the second modification, which is that, due to \propref{Prop:NewJhx}, the relation \eqref{hij=vij} is to be replaced with 
\begin{equation*}
 [h_{ij}(\alpha),x_\alpha^\mp]
 =[(\alpha,\alpha_j)J(h_i)-(\alpha,\alpha_i)J(h_j),x_\alpha^\mp]-[v_{ij}(\alpha),x_\alpha^\mp]=[x_\alpha^\mp,v_{ij}(\alpha)]\mp d_{\alpha,i,j}\frac{\ve}{2}x_\alpha^\mp.
\end{equation*}
In particular, the right-hand side of \eqref{suff:2'} is still equal to $2\sum_{\alpha\in \Delta_+^{\mathrm{re}}(k)}[x_\alpha^\mp,v_{ij}(\alpha)]\otimes x_\alpha^\pm$, and the remainder of the proof thus proceeds without modification. Therefore, we may conclude that \thmref{thm:coproduct} holds for $Y_{\hbar=1,\ve}(\g)$.

Finally, the results on completions in \secref{Sec:comp} also hold for $Y_{\hbar=1,\ve}(\g)$, so in particular we may view $\Delta$ as an algebra homomorphism $Y_{\hbar=1,\ve}(\g)\to Y_{\hbar=1,\ve}(\g)\widehat \otimes Y_{\hbar=1,\ve}(\g)$. Moreover, we can recover $\Delta$ and $\Delta_{V_1,V_2}$ from a parameter dependent coproduct $\Delta_u$ as in \secref{Sec:paracoprod}.

\bibliographystyle{hamsalpha}
\bibliography{coprodbib}

\end{document}